\documentclass[reqno,makeidx,12pt]{amsart}
\usepackage[T1]{fontenc}
\usepackage[utf8]{inputenc}  
\usepackage{graphicx}
\usepackage{latexsym}
\usepackage{amstext}
\usepackage {amsmath}
\usepackage {amsfonts}
\usepackage {amssymb}
\usepackage {amsthm}
\usepackage {bbm}
\usepackage{esint}
\usepackage{enumerate}
\usepackage{array}
\usepackage{comment}
\usepackage{xspace}
\usepackage[bookmarks=true]{hyperref}
\usepackage{enumerate}
\ifx\cs\documentclass \footheight 12pt \fi \footskip 30pt 
\DeclareMathAlphabet{\mathpzc}{OT1}{pzc}{m}{it}
\newcommand{\Hardy}{\ensuremath{\mathcal{H}}}
\newcommand{\BMO}{\ensuremath{\text{BMO}}}

\DeclareMathOperator\id{Id}

\DeclareMathOperator{\dive}{div}

\DeclareMathOperator{\A}{A}
\DeclareMathOperator{\arccosh}{arccosh}
\DeclareMathOperator{\graph}{graph}
\DeclareMathOperator{\tr}{tr}
\DeclareMathOperator{\spt}{spt}
\DeclareMathOperator{\cof}{cof}
\DeclareMathOperator{\osc}{osc}
\DeclareMathOperator{\Lip}{Lip}

\newcommand{\R}{\ensuremath{\mathbb{R}}}

\newcommand{\N}{\ensuremath{\mathbb{N}}}
\newcommand{\LL}{\ensuremath{\mathcal{L}}}
\newcommand{\E}{\ensuremath{\mathcal{E}}}

\newcommand{\Ha}{\ensuremath{\mathcal{H}}}

\def\cringle{\mathaccent"7017 }
\newcommand{\Atf}{\ensuremath{\cringle{\A}}}

\newcommand{\ar}{ar}

\newcommand{\sgp}{\Pi}
\newcommand{\ctf}{\vartheta}
\newcommand{\D}{\mathcal{D}}
\newcommand{\nbr}{\#}

\def\R{\mathbb R}
\def\Z{\mathbb Z}

\def\W{\mathcal{W}}

\def\M{\mathcal{M}}
\def\opt{\text{w}}
\def\step#1{\mbox{}\\\underline{\text{Step #1:}}}%

\def\XXint#1#2#3{{\setbox0=\hbox{$#1{#2#3}{\int}$ }
\vcenter{\hbox{$#2#3$ }}\kern-.6\wd0}}

\theoremstyle{note}
\newtheorem{theorem}{Theorem}
\newtheorem{corollary}{Corollary}
\newtheorem{lemma}{Lemma}
\newtheorem{proposition}{Proposition}

\theoremstyle{definition}

\newtheorem*{remark-short}{Remark}

\begin{document}
\title{Confined structures of least bending energy}
\author{Stefan M\"uller}
\address{Stefan M\"uller,  Institute for Applied Mathematics and Hausdorff Center for Mathematics, University of Bonn, Endenicher Alllee 60, D-53115 Bonn, Germany}
\email{sm@hcm.uni-bonn.de}

\author{Matthias R{\"o}ger}
\address{Matthias R{\"o}ger, Technische Universit\"at Dortmund,
Fakult\"at für Mathematik,
Vogelpothsweg 87,
D-44227 Dortmund}
\email{matthias.roeger@tu-dortmund.de}



\date{\today}

\begin{abstract}
In this paper we study a constrained minimization problem for the Willmore functional. For prescribed surface area we consider smooth embeddings of the sphere into the unit ball. We evaluate the dependence of the  the minimal Willmore energy of such surfaces on the prescribed surface area and prove corresponding upper and lower bounds. Interesting features arise when the prescribed surface area just exceeds the surface area of the unit sphere. We show that (almost) minimizing surfaces cannot be a $C^2$-small perturbation of the sphere. Indeed they have to be nonconvex and there is a sharp increase in Willmore energy with a square root rate with respect to the increase in surface area.
\end{abstract}
%\fbox{\Large \centerline{DRAFT, not for circulation!}}\ \\[5mm]
\maketitle
%\tableofcontents
%==========================================
\section{Introduction}
%==========================================
\label{sec:intro}
Constrained minimization problems for bending energies arise naturally in various applications. In biophysics for example the shape of the cell membranes is often modeled as (local) minimizer of an appropriate curvature energy, most notably of the Helfrich--Canham energy
\begin{gather*}
	\E_{HC}(\Sigma)\,=\, \int_\Sigma \big(\kappa_b (H-H_0)^2 +\kappa_g K\big)\,d\Ha^2.
\end{gather*}
Here $\Sigma\subset\R^3$ is a smooth surface describing the shape of the cell, $H$ and $K$ are the mean and Gaussian curvature of $\Sigma$, and the spontaneous curvature $H_0$ and the bending moduli $\kappa_b,\kappa_g$ are given parameters. Under appropriate constraints on the total surface area and on the enclosed volume local minimizers of such shape energies are in good agreement with typical shapes of cell membranes.

In this article we are interested in the minimization of bending energies under an additional {confinement} condition. This problem is motivated by the shape of inner organelles in a biological cell. These structures are confined to the inner volume of the cell. Moreover, as the membrane contributes to their biological function, organelles  often have large surface area (see for example the typical shape of mitochondriae).

We start here a mathematical analysis of a simple prototype of such constrained minimization problems: As curvature energy we consider the Willmore functional and we choose as outer container the unit ball. To give a precise description of the problem let us introduce some notation: Let $a>0$ be given and let $B=B(0,1)$ be the unit ball in $\R^3$. We denote by $\M_{a}$ the class of smoothly embedded surfaces $\Sigma\subset B$ of sphere type with surface area $\ar(\Sigma)=a$. We associate to $\Sigma\in\M_a$ the outer unit normal field $\nu:\Sigma\to\R^3$, denote by  $\kappa_1, \kappa_2$ the principal curvatures of $\Sigma$ with respect to $\nu$, and define the scalar mean curvature $H=\kappa_1+\kappa_2$, the mean curvature vector $\vec{H}=-H\nu$, and the Gauss curvature $K=\kappa_1\kappa_2$.

For $\Sigma\in\M$ we then consider the Willmore energy
\begin{gather}
	\W(\Sigma)\,:=\, \frac{1}{4}\int_\Sigma |\vec{H}|^2\,d\Ha^2 \label{eq:willm}
\end{gather}
and the constrained minimization problem
\begin{gather}\label{eq:def-min-cs}
  \opt(a)\,:=\, \inf_{\Sigma\in\M_{a}} \W(\Sigma).
\end{gather}
We are interested in  the dependence of $\opt(a)$ on the surface area $a$, in particular for large values of $a$. The infimum $\opt(a)$ may not be attained, as limit points of minimal sequences need not to be embedded. Therefore, we can not make use of the Euler--Lagrange equation. It is an interesting open problem to identify a class of (generalized) surfaces that comprises the closure of $\M_a$ and in which the infimum of the Willmore energy is attained. One possible candidate is the class of Hutchinson varifolds that have a unique tangent plane in every point but possibly varying multiplicity.

Our main results are first a general lower bound $\opt(a)\geq a$ and the optimality of this bound for $a=4k\pi$ with $k\in\N$, and second a characterization of the behavior of $\opt$ as $a$ just exceeds the value $4\pi$. For $a=4k\pi$ the optimal value $\opt$ realizes the Willmore energy of $k$ spheres and the varifold limit of a minimal sequence converges to the unit sphere with density $k$. Configurations at $a\approx 4\pi k$ resemble $k$ unit spheres (connected by catonoid like structures in order to have the topology of a sphere). We therefore believe that the behavior of $\opt$ as $a$ crosses $4\pi$ is key for the understanding of the constrained minimization problem. As there are no surfaces that are $C^2$-close to the sphere with area above $4\pi$ a change of behavior at this value can be expected. In fact, we prove a sharp increase in the optimal energy at $4\pi$: the difference in Willmore energy $\opt(a)-4\pi$ behaves like the square root of the area difference $a-4\pi$. The proof of the corresponding lower bound is the most delicate step and uses rigidity estimates for nearly umbilical surfaces shown by De Lellis and Müller \cite{DeMu05,DeMu06}.

Whereas our analysis does use the particular choice of the unit ball as the confinement condition we also gain some insight in the minimization problem for more general containers $C\subset\R^3$. In particular we obtain general upper and lower bounds that are linear in $a$. In fact, if $C\subset B(x_0,R)$ a rescaling argument shows that $\W(\Sigma)\geq \frac{a}{R^2}$ for any $\Sigma\subset  C$. If on the other hand $B(x_1,r)\subset C$ then $\W(\Sigma)\leq 4\pi k$ for any $a=4\pi k r^2$, $k\in\N$. In case of a convex container with $C^2$-boundary we expect that with growing surface area first the full space provided by the container will be used (with a linear growth rate of the minimal Willmore energy) before a protrusion inside the container will be developed (with a square root type increase in Willmore energy).
Comparing the behavior of our constrained minimization problem with the shape of inner structures in cells we remark that our model rather supports formation of single protrusions that grow inside than the formation of multiple folds. This indicates that for a proper model of such structures more details have to be taken into account such as the dynamic process of fold formation or additional constraints on the enclosed volume of the inner structures.

The minimization of the Willmore functional under constraints has been studied in detail for rotationally symmetric surfaces, see \cite{Seif97} for a review. General existence results without any symmetry assumptions were obtained by Simon \cite{Simo93} proving the existence of smooth minimizer for the Willmore functional for tori in $\R^3$. This result was extended  to surfaces with arbitrary prescribed genus by Bauer and Kuwert \cite{BaKu03}. Recently Schygulla \cite{Schy12} showed the existence of smooth minimizers of the Willmore functional for sphere-type surfaces with prescribed isoperimetric ratio. One estimate that was shown in \cite{Simo93} and has been refined in \cite{Topp98} is the following relation between Willmore functional, surface area, and (external) diameter $d$,
\begin{gather*}
	\W(\Sigma) \,\geq\, \frac{d^2\pi^2}{4\ar(\Sigma)}.
\end{gather*}
For our purposes however, this estimates is not very helpful as it degenerates with increasing surface area. An alternative approach for minimizing the Willmore energy is to employ a gradient flow. For the Willmore flow Simonett \cite{Simo01} and Kuwert and Schätzle \cite{KuSc01,KuSc02,KuSc04} have proved existence and convergence results. However, as we need to satisfy constraints on area and confinement such results are not directly applicable to our problem. 

A closely related confinement problem has been studied by numerical simulations in \cite{KaSM12}. For a phase field approach to the minimization of the Willmore energy under a confinement and connectedness constraint see \cite{DoMR13}.
%=================================
%\subsection*{Acknowledgement}
%=================================
\section{Estimate from below}
We will first prove a general lower bound for surfaces in the unit ball by exploiting the classical Gauss integration by parts formula on manifolds. As remarked above, limit points of minimizing sequences for our constrained minimization problem may leave the class $\M_a$. By Allard's compactness theorem \cite{Alla72} such limit points at least belong to the class of integral 2-varifolds with weak mean curvature in $L^2$, see \cite{Simo83} for the relevant definitions (note that we identify an integral 2-varifold with its associated weight measure on $\R^3$). It is therefore useful (and straigthforward) to prove the lower bound in this extended class of generalized surfaces.
\begin{theorem}\label{thm:est-below}
Let $\mu$ be an integral 2-varifold with weak mean curvature vector $\vec{H}\in
L^2(\mu)$ and support contained in $\overline{B}$. Then we have
\begin{gather}\label{eq:est-below}
  \int \frac{1}{4}|\vec{H}|^2\,d\mu \,\geq\, \mu(\overline{B})
\end{gather}
and equality holds if and only if $\mu=k\Ha^2\lfloor S^2$ for an integer
$k\in\N$. 
\end{theorem}
\begin{proof}
Since $\mu$ has weak mean curvature $H\in L^2(\mu)$ we have (just by definition of weak mean curvature)  that for any $\eta\in C^1_c(\R^3;\R^3)$ the first variation formula
\begin{gather*}
	\int \dive_{T_x\mu}\eta(x)\,d\mu(x)\,=\, -\int  \vec{H}(x)\cdot x\,d\mu(x)
\end{gather*}
holds. Consider now the vector field
$\eta(x):=x$. Then $\dive_{T_x\mu}\eta(x)=2$ and we deduce
\begin{align*}
  2\mu(\overline{B})\,&=\, \int \dive_{T_x\mu}\eta(x)\,d\mu(x)\\
  &=\, -\int
  \vec{H}(x)\cdot x\,d\mu(x) \\
  &=\, \int \frac{1}{4} |\vec{H}|^2\,d\mu(x) +\int 1
  \,d\mu(x)-\frac{1}{4}\int \big|\vec{H}+2x^\perp\big|^2\,d\mu(x) \\
  &\qquad - \int \big(1-|x^\perp|^2\big)\,d\mu(x),
\end{align*}
where for $x\in \spt(\mu)$ the projection onto $(T_x\mu)^\perp$ is denoted by $x^\perp$ and where we have used that $\vec{H}(x)$ is perpendicular to $T_x\mu$ in $\mu$-almost every point \cite[Thm 5.8]{Brak78}. 
From the last equality we obtain
\begin{align}
  \mu(\overline{B}) &=\, \int \frac{1}{4} |\vec{H}|^2\,d\mu(x) -\frac{1}{4}\int \big|\vec{H}+2 x^\perp\big|^2\,d\mu(x) 
   - \int \big(1-|x^\perp|^2\big)\,d\mu(x). \label{eq:est-below-1} 
\end{align}
Since $|x|\leq 1$ this immediately implies \eqref{eq:est-below}.
Equality in \eqref{eq:est-below} holds if and only if $x^\perp=-\frac{1}{2}\vec{H}(x)$ and $|x^\perp|=1$ for  $\mu$-almost every $x\in \overline{B}$. This implies $x=x^\perp $ and $|x|=1$ for $\mu$-almost every point $x\in\overline{B}$, in particular $\spt(\mu)\subset S^2$. 
From the monotonicity formula one derives \cite[(A.17)]{KuSc04} that for any $x_0\in S^2$  the two-dimensional density satisfies $\theta^2(\mu,x_0)\leq \frac{1}{4\pi}\W(\mu)$. Therefore, if equality holds in \eqref{eq:est-below}, then
\begin{gather*}
	\W(\mu) \,=\, \mu(\overline{B})\,=\, \int_{S^2} \theta^2(\mu,x_0) \,d\Ha^2(x_0)\,\leq\, \W(\mu)
\end{gather*}
and we thus obtain that $\theta^2(\mu,\cdot)=\frac{1}{4\pi}\W(\mu)$ hold $\mu$-almost everywhere. By integrality of $\mu$ this in particular implies $\mu=  k \Ha^2\lfloor S^2$.
\end{proof}
This result immediately implies a lower bound for $\opt$ and shows that equality can only be attained for $a=4k\pi$ with $k\in\N$.
\begin{corollary}\label{cor:est-below}
We have
\begin{gather*}
  \opt(a)\,\geq\, a\quad \text{ for all }a>0,\\
  \opt(a)\,>\, a\quad\text{ for all }a\in \R_+\setminus\{4k\pi
  \,:\,k\in\N\}.
\end{gather*}
\end{corollary}
\begin{proof}
Let $a\in\R$ be fixed and $(\Sigma_j)_{j\in\N}$ be a minimal sequence in
$\M_a$. We associate with $\Sigma_j$ the integer rectifiable varifolds $\mu_j=\Ha^2\lfloor
\Sigma_j$. For all $j\in \N$ the varifold $\mu_j$ has total mass $\mu_j(\overline{B})=a$ and 
mean curvature vector $\vec{H}_j$ that is uniformly bounded in $L^2(\mu_j)$ by $\|\vec{H}_j\|_{L^2(\mu_j)}^2\,\leq\,4\opt(a)+1$. By
Allards compactness theorem for integral varifolds \cite{Alla72}  there exists a
subsequence of $\mu_j$ that converges to an integral varifold $\mu$ with weak mean
curvature $\vec{H}\in L^2(\mu)$. In addition the support of $\mu$ is contained in $\overline{B}$ and we have
\begin{gather*}
  \mu(\overline{B})\,=\,\lim_{j\to\infty} \mu_j(\overline{B})\,=\, a.
\end{gather*}
Furthermore we obtain that for any $\eta\in C^1_c(\overline{B})$
\begin{align*}
  \int \vec{H}\cdot\eta\,d\mu \,&=\, -\int
  \dive_{T_x\mu}\eta(x)\,d\mu(x) \,=\, -\lim_{j\to\infty}\int_{\Sigma_j}
  \dive_{T_x\Sigma_j}\eta(x)\,d\Ha^2(x) \\
  &=\, \lim_{j\to\infty}\int_{\Sigma_j}
  \vec{H}_j\cdot\eta\,d\Ha^2 \\
  &\leq\, \liminf_{j\to\infty}\Big(\int_{\Sigma_j}
  \eta^2\,d\Ha^2(x)\Big)^{1/2}\Big(\int_{\Sigma_j}
  |\vec{H}_j|^2\,d\Ha^2\Big)^{1/2}\\
  &=\, \Big(\int
  \eta^2\,d\mu\Big)^{1/2}\liminf_{j\to\infty}\Big(\int_{\Sigma_j}
  |\vec{H}_j|^2\,d\Ha^2\Big)^{1/2}
\end{align*}
and it follows that
\begin{gather*}
  \int \frac{1}{4}|\vec{H}|^2\,d\mu \,=\, \frac{1}{4}\Big(\sup_{\|\eta\|_{L^2(\mu)}\leq 1}
  \int \vec{H}\cdot\eta\,d\mu\Big)^2\,\leq\, \liminf_{j\to\infty}
  \W(\Sigma_j)\,=\, 
  \opt(a).
\end{gather*}
Theorem \ref{thm:est-below} then first yields $\opt(a)\geq \mu(\overline{B})=a$ and secondly that $\opt(a)=a$ implies
$\mu=k\Ha^2\lfloor S^2$ for an $k\in\N$ and $\mu(\overline{B})=4k\pi $.
\end{proof}

%==========================================
%\section{The minimal energy for $A=4\pi k$}
%==========================================
We next show that in fact for $a= 4k\pi, k\in\N,$ the optimal value $\opt(a)=a$ is achieved.
\begin{theorem}\label{thm:m-sphere}
Let $a=4k\pi$ for $k\in\N$. Then $\opt(a)=a$ and any minimizing
sequence converges as varifolds to $\mu=k\Ha^2\lfloor S^2$.
\end{theorem}
\begin{proof}
The last property is proved by similar arguments as used in the proof of
Corollary \ref{cor:est-below}. To show that $\opt(a)=a$ holds we
construct a  sequence $(\Sigma_j)_{j\in\N}\subset \M_a$ such that
$\W(\Sigma_j)\to a$. For $k=1$ the unit sphere is the unique minimizer.
The main idea for $k=2$ is to take two concentric spheres, one with
radius one and the 
other with radius close to one. For both spheres we remove a cap close to
the north-pole, deform the upper halves, and connect them by a catenoid-like
structure, see Figure \ref{fig:const1}.  We give the details of the
proof in Section \ref{sec:proof}.
\begin{figure}[h]
\begin{minipage}{\textwidth}
\begin{center}
  \begin{minipage}[b]{.4\linewidth} % [b] => Ausrichtung an \caption
  \includegraphics*[width=6cm]{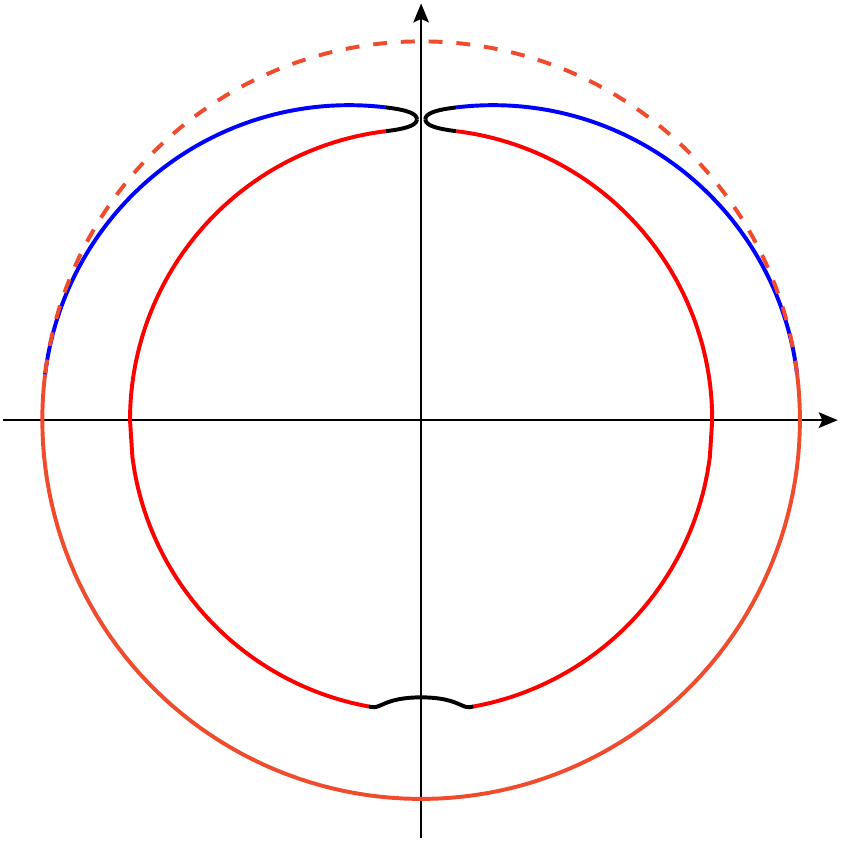}
%  \caption{The construction for moderate $j$}
%  \label{fig:const1}
  \end{minipage}
  \hspace{.1\linewidth}% Abstand zwischen Bilder
  \begin{minipage}[b]{.4\linewidth} % [b] => Ausrichtung an \caption
  \includegraphics*[width=6cm]{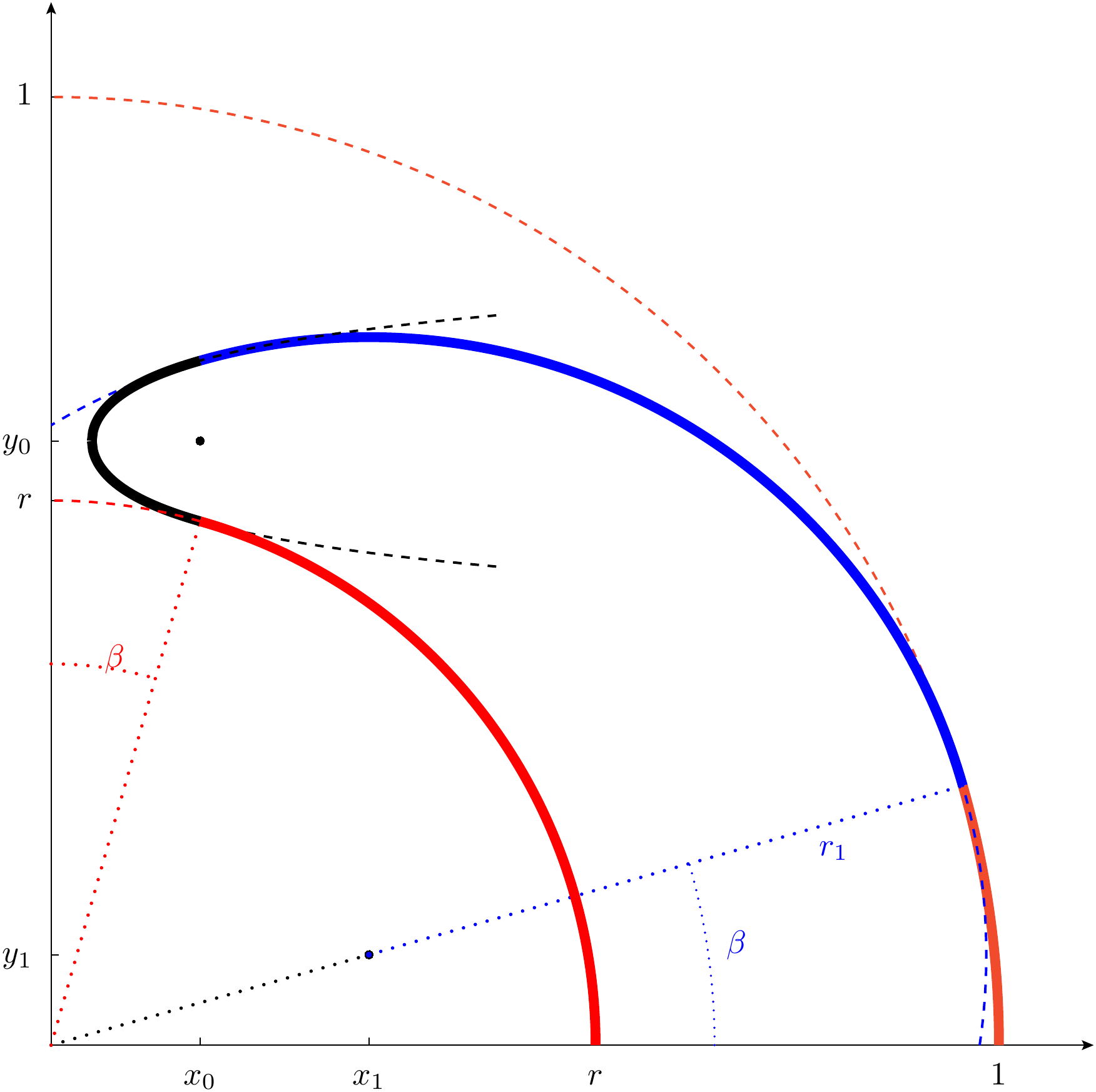}
  \end{minipage}
\end{center}
  \caption{Construction of a minimizing sequence. For details we refer to the Appendix.}
  \label{fig:const1}
\end{minipage}
\end{figure}

For $k\geq 3$ we take $k$ nested spheres and apply $(k-1)$-times the
construction described for $k=2$. 
\end{proof}
%==========================================
\section{Upper bound for $a$ close to $4\pi k$}
%==========================================
Using that a dilation of space does not change the Willmore energy we
obtain the following monotonicity property.
\begin{proposition}\label{prop:A_below}
The mapping $a\mapsto \opt(a)$ is monotonically increasing. In particular, for all $a\leq 4\pi k$ we have
\begin{gather}\label{eq:A_below}
  \opt(a)\,\leq\, 4\pi k.
\end{gather}
\end{proposition}
\begin{proof}
Fix $0<a_1<a_2$ and let $(\tilde{\Sigma}_j)_{j\in\N}$ be a minimal sequence in $\M_{a_2}$,
\begin{gather*}
  \ar(\tilde{\Sigma}_j)\,=\, a_2 \quad\text{ for all }j\in\N,\qquad
  \lim_{j\to\infty} \W(\tilde{\Sigma}_j)\,\to\, \opt(a_2).
\end{gather*}
Let $s:= \sqrt{\frac{a_1}{a_2}}<1$ and denote by $\vartheta_s:\R^3\to\R^3$ the
dilation by factor $s$, i.e. $\vartheta_s(x)=sx$. Define
\begin{gather*}
  \Sigma_j\,:=\, \vartheta_s(\tilde{\Sigma}_j).
\end{gather*}
Then $\ar(\Sigma_j)\,=\, s^2\ar(\tilde{\Sigma}_j)\,=\, a_1$ and $\Sigma_j\in\M_{a_1}$ for all $j\in\N$. Moreover
\begin{gather*}
  \W(\Sigma_j)\,=\,\W(\tilde{\Sigma}_j)\,\to\, \opt(a_2)
\end{gather*}
and therefore $\opt(a_1)\leq\opt(a_2)$. Since $\opt(4\pi k)=4\pi k$ by Theorem \ref{thm:m-sphere} the second conclusion follows.
\end{proof}
For $k=1$ the sphere with radius $r(a):=\sqrt{a/(4\pi)}$ is the unique
minimizer of $\W$ in $\M_a$ (up to translations) and \eqref{eq:A_below}
is sharp. %\meta{Is it sharp also for $k>1$ ?}

For $a$ approaching $4\pi $ from above we have the following upper bound.
\begin{proposition}\label{prop:A-above}
For all $\delta>0$ there exists a constant $C>0$ such that
\begin{gather}\label{eq:est-A-above}
  \opt(a) - 4 \pi k\,\leq\, C\cdot \sqrt{a-4\pi }
\end{gather}
for all $4\pi k\leq a< 4\pi k+\delta$, $k\in\N$.
\end{proposition}
\begin{proof}
The first five steps of the proof deal with the case $k=1$.
\step{1}
We modify the unit sphere by growing a `bump', directed inwards and
supported close to $(0,0,1)$. First we choose two parameters $0<s,t\ll 1$
controlling the support of the bump and its extension. We fix
a symmetric function $\eta\in C^\infty(-1,1)$ that is positive inside
its support and  decreasing on $(0,1)$. We define
\begin{gather*}
  \psi=\psi_{s,t}: [0,1]\,\to\,\R,\qquad \psi(r)\,:=\, \sqrt{1-r^2}
  -t\eta(rs^{-1}). 
\end{gather*}
Next let
\begin{gather*}
  \Psi: B^2_1(0)\,\to\, \R,\qquad \Psi(x)=\psi(|x|)
\end{gather*}
and define
\begin{gather*}
  M_{s,t}\,:=\, \graph \Psi,\qquad \Sigma_{s,t}\,:=\, M_{s,t}\cup S^2_-,
\end{gather*}
where $S^2_-$ denotes the lower half of the unit sphere.
Then, for $t<t_0(\eta)$, the surface $\Sigma_{s,t}$ 
is smooth, compact, without boundary,
and is contained in $B_1^3(0)$. Moreover, we have $\Sigma_{s,0}=
S^2$.
\step{2}
We compute the surface area element $g$, the scalar mean curvature $H$, and the
Gaussian curvature $K$ of $M_{s,t}$. We first obtain
\begin{align}
  \psi'(r)\,&=\, -\frac{r}{\sqrt{1-r^2}} -
  \frac{t}{s}\eta'(rs^{-1}),\\
  \psi''(r) \,&=\, -(1-r^2)^{-3/2} -\frac{t}{s^2}\eta''(rs^{-1}).
\end{align}
For the surface area element $g(x)=g(r)$ we deduce
\begin{gather}
  g(r)^2\,=\, 1+ |\nabla \psi|^2\,=\, \frac{1}{1-r^2} +
  2\frac{r}{1-r^2}\frac{t}{s}\eta'(rs^{-1})
  +\frac{t^2}{s^2}\eta'(rs^{-1})^2. \label{eq:comp-g}
\end{gather}
For the scalar mean curvature $H(r)=H(x)=-\nabla\cdot
\big(\nabla\psi(x)/g(x)\big)$ we have 
\begin{align}
  g(r)^3 H(r)\,=&\, 2(1-r^2)^{-3/2} + \frac{t}{s^2}\eta''(rs^{-1}) +
  \frac{t}{s}\frac{1}{r} \eta'(rs^{-1}) + 3
  \frac{t}{s}\frac{r}{1-r^2}\eta'(rs^{-1}) \notag\\
  &+ 3\frac{t^2}{s^2}\frac{1}{\sqrt{1-r^2}}\eta'(rs^{-1})^2 +
  \frac{1}{r}\frac{t^3}{s^3} \eta'(rs^{-1})^3. \label{eq:comp-H}
\end{align}
\step{3}
We choose $t$ in dependence of $s$ such that $M_{s,t}$ has larger area
than the half-sphere and such that the area converges to $2\pi$ as $s\to
0$.

We first observe that $M_{s,t}$ only differs from $S^2\cap\{x_3>0\}$ in
$B_s^2(0)\times\R$. Therefore, by a Taylor expansion of the square root
in $g(r)$,
\begin{align}
  &\ar(M_{s,t})-\ar(S^2\cap\{x_3>0\})\notag\\
  =\, &\int_0^s 2\pi r g(r)\,dr -
  \int_0^s 2\pi r \frac{1}{\sqrt{1-r^2}}\,dr \notag\\
  =\, &2\pi \int_0^s
  \frac{r}{2}\sqrt{1-r^2}\Big(2\frac{t}{s}\frac{r}{\sqrt{1-r^2}}\eta'(rs^{-1}) +
  \frac{t^2}{s^2}\eta'(r/s)^2\Big)\, dr + 2\pi\int_0^s r R_{s,t}(r)\,dr,
  \label{eq:diff-area1} 
\end{align}
where $|R_{s,t}(r)|\,\leq\, C\Big|2\frac{t}{s}\frac{r}{\sqrt{1-r^2}} +
\frac{t^2}{s^2}\eta'(r)^2\Big|^2$. For $t\ll s$ we therefore can approximate
\begin{align}
  &\ar(M_{s,t})-\ar(S^2\cap\{x_3>0\}) \notag\\
  \approx\,&
  2\pi \int_0^s\Big(
  \frac{t}{s}r^2\eta'(rs^{-1}) +
  \frac{t^2}{s^2}\frac{r}{2}\sqrt{1-r^2}\eta'(rs^{-1})^2 \Big)\,dr
  \notag\\
  \approx\,& 2\pi t\int_0^1 \Big(s^2\varrho^2\eta'(\varrho) +
  t\frac{\varrho}{2}\eta'(\varrho)^2\Big)\, d\varrho.
  \label{eq:diff-area2} 
\end{align}
We now can choose $\alpha\ll1$ depending only on $\eta$ such that the
for $t=\alpha s^2$ the right-hand side is positive and converges to zero
with $s\to 0$, more precisely
\begin{gather}
  \ar(M_{s,t})-\ar(S^2\cap\{x_3>0\}) 
  \approx\, 2\pi \alpha s^4 C(\eta)\,>\,0. \label{eq:asym-area}
\end{gather}
\step{4}
We next show that the mean curvature is uniformly bounded in
$s>0$. Since $g(r)\approx 1$ it is sufficient to bound the right-hand
side of \eqref{eq:comp-H}. We estimate the different terms.
\begin{align*}
  &2(1-r^2)^{-3/2} \,\approx\, 2,\\
  &|\frac{t}{s^2}\eta''(rs^{-1})| \,\leq\, \alpha
  \|\eta''\|_{C^0},\\
  &0\,\geq\, \frac{t}{s}\frac{1}{r}\eta'(rs^{-1})\,\geq\, -\alpha
  \frac{s}{r}\big(\eta'(rs^{-1})-\eta'(0)\big)\,\geq\,
  \alpha\|\eta''\|_{C^0},\\
  &0\,\geq\, \frac{t}{s}\frac{r}{1-r^2}\eta'(rs^{-1})\,\geq\,
  \alpha\frac{sr}{1-r^2} \eta'(rs^{-1})\,\geq\, -
  2\alpha s^2\|\eta'\|_{C^0},\\
  &0\,\leq\,\frac{t^2}{s^2}\frac{1}{\sqrt{1-r^2}}\eta'(rs^{-1})^2
  \,\leq\, \frac{1}{2}\sqrt{2}\alpha^2 s^2\|\eta'\|_{C^0}^2,\\ 
  &0\,\geq\,\frac{t^3}{s^3}\frac{1}{r}\eta'(rs^{-1})\,\geq\,-r^2\alpha^3\|\eta''\|_{C^0}^3
  \,\geq\,-s^2\alpha^3\|\eta''\|_{C^0}^3.
\end{align*}
Together with \eqref{eq:comp-H} this yields
\begin{gather}
  |H(r)|\,\leq\, C(\eta). \label{eq:est-H}
\end{gather}
\step{5}
By the construction above we obtain a sequence $s\to 0$ and smooth,
compact surfaces $\Sigma_s$ without boundary
and contained in $B_1^3(0)$, such that 
\begin{gather*}
  a(s)\,:=\,\ar(\Sigma_s)\,>\,4\pi,\qquad a(s)\,\to\, 4\pi \,(s\to 0)
\end{gather*}
and
\begin{align*}
  &\opt(a(s))-  4\pi \\
  \leq\,&
  \W(\Sigma_s)-\W(S^2) \\
  \leq\,&
  \sup\{|H(r)|^2:0<r<s\}\ar\big(\graph(\Psi|_{B_s^2(0)})\big)
  -  \ar(\graph|_{B_s^2(0)}(r\mapsto \sqrt{1-r^2}))\notag\\
  \leq\,& C(\eta) (a(s)-4\pi) +
  |2-C(\eta)|\ar(\graph|_{B_s(0)}(r\mapsto \sqrt{1-r^2}))\\
  \leq\, & C(\eta)\Big((a(s)-4\pi) + s^2\Big)\,\leq\, C(\eta)\sqrt{(a(s)-4\pi)}
\end{align*}
by \eqref{eq:asym-area}.
\step{6}
For $k\geq 2$ we follow the construction of a minimal sequence for
$a=4k\pi $ described in Section \ref{sec:proof}, except that we grow in
Step~5 of Section \ref{sec:proof} a slightly larger bump, such that the area of
the constructed surface just exceeds $4k\pi $.
\end{proof}
%====================
% lower bound
%====================
\section{Lower bound for $a$ close to $4\pi$}
\label{sec:lb-4pi}
By Corollary \ref{cor:est-below} we immediately obtain the lower estimate
\begin{gather}
	\opt(a)-\opt(4\pi)\,\geq\, a-4\pi \label{eq:lb-lin}
\end{gather}
for $a\geq 4\pi$. The upper bound in Proposition \ref{prop:A-above} on the other hand shows the square-root behavior $\opt(a)-\opt(4\pi)\,\leq\,C\sqrt{ a-4\pi}$. In this section we derive an improved lower bound with square-root type growth rate.

 The next Proposition gives a  a useful characterization of the area difference. In particular we see that there are no surfaces in $\M_a$ with area larger
than $4\pi$ that are $C^2$-close to the sphere, which gives a first hint to a change of behavior in the constrained optimization.
\begin{proposition}\label{prop:4pi-est-below}
For any $\Sigma\in\M_a$ 
\begin{gather}
  \ar(\Sigma) -4\pi\,=\, -\int_\Sigma \Big(1-(x\cdot\nu(x))^2 +
  \frac{1}{2}|x- (x\cdot\nu(x))\nu(x)|^2\Big) K(x)\,d\Ha^2(x) \label{eq:K-M}
\end{gather}
holds.
In particular, $K\geq 0$ on $\Sigma$ implies $\ar(\Sigma)\leq 4\pi$.
\end{proposition}
\begin{proof}
Let $\nu$ denote a smooth unit-normal field on $\Sigma$, and let $(e_1,e_2,e_3)=(\tau_1,\tau_2,\nu)$ be a smooth orthonormal frame on $\Sigma$. We define
$p_i(x):=x\cdot \tau_i$, $q:=\sqrt{p_1^2+p_2^2}$, $\eta(x):=x$, and $\omega_{ij} := e_j\cdot de_i$. For the
1-form $\omega:= \eta\cdot(\nu\times d\nu)= p_2\omega_{31}-p_1\omega_{32}$ we compute \cite[proof of Theorem 26]{AgFr02}
\begin{gather}
  d\omega\,=\, -H\sigma + 2p_3 K\sigma,\quad dp_3\,=\, \omega_{31}p_1 + \omega_{32}p_2,
  \label{eq:omega}
\end{gather}
where $\sigma$ denotes the volume form on $\Sigma$ (note that in \cite{AgFr02} the mean curvature is defined as $\frac{1}{2}$ times the trace of the Weingarten map, hence  the term $2H$ appears there instead of $H$). We thus obtain
\begin{gather*}
	d(p_3\omega)\,=\, dp_3\wedge \omega + p_3 d\omega \,=\, -q^2 K\sigma -p_3 H\sigma + 2p_3^2 K\sigma.
\end{gather*}
Integration over $\Sigma$ yields
\begin{gather}
  \int_\Sigma \big(p_3^2 - \frac{1}{2}q^2\big)K\,d\sigma \,=\,
  \frac{1}{2}\int_\Sigma p_3 H\sigma\,=\, \frac{1}{2}\int_\Sigma
  -x\cdot\vec{H}\sigma\,=\, \frac{1}{2}\int_\Sigma \dive_{\Sigma}\eta \sigma \,=\, \ar(\Sigma), \label{eq:K-M-1}
\end{gather}
where we have used in the last two equalities the classical divergence formula on smooth closed surfaces \cite[(7.6)]{Simo83} and $\dive_{\Sigma}\eta=2$ on $\Sigma$.\\
By the Gauss-Bonnet formula $\int_\Sigma K\,d\Ha^2\,=\, 4\pi$ and
substracting 
this identity from \eqref{eq:K-M-1} we obtain \eqref{eq:K-M}.
\end{proof}
The main result of this section is following improved lower bound.
\begin{theorem}\label{thm:lb-4pi}
There exists $c>0$ such that for all $\Sigma\in\M_a$, $a\geq 4\pi$ 
\begin{gather}
	\opt(a)-4\pi\,\geq\, c\sqrt{a-4\pi} \label{eq:lb-4pi}
\end{gather}
holds.
\end{theorem}

In the remainder of this section we prove Theorem \ref{thm:lb-4pi}. We first introduce some notation and recall rigidity estimates for nearly umbilical surfaces derived by De Lellis and Müller \cite{DeMu05,DeMu06}.

For $\Sigma\in \M_a$ let $g$ denote the first fundamental form of $\Sigma$,  $\nu$ the outer unit normal field, $A$ the second fundamental form, $A(v,w)=g(v,d\nu(w))$, $H=\tr(A)$. With this convention the unit sphere has $H=2$ and $A=\id$. Let further $\Atf$ denote the trace-free part of the second fundamental form,
\begin{gather*}
	\Atf(x)\,=\, A(x)-\frac{\tr{A(x)}}{2} \otimes g\,=\, A(x)-\frac{1}{2}H(x)\otimes g\qquad\text{ for }x\in\Sigma.
\end{gather*}
We have the relation
\begin{gather*}
	2|\Atf|^2\,=\, \kappa_1^2 +\kappa_2^2 -2\kappa_1\kappa_2 \,=\, H^2 -4 K.
\end{gather*}
The Gauss--Bonnet Theorem then implies that
\begin{gather}
	\W(\Sigma)-\W(S^2)\,=\, \frac{1}{4}\int_\Sigma H^2 \,d\Ha^2 -4\pi \,=\, \frac{1}{2}\int_\Sigma |\Atf|^2\,d\Ha^2. \label{eq:atf}
\end{gather}
By \cite[Theorem 1.1]{DeMu05}  for $\Sigma\in \M_{4\pi}$ with $\W(\Sigma)\leq 6\pi$ there exists a universal constant $C>0$ and  a conformal parametrization $\psi:S^2\to\Sigma$ such that after a suitable translation
\begin{gather}
	\|\psi -\id\|_{W^{2,2}(S^2)}\,\leq\, C\|\Atf\|_{L^2(\Sigma)}. \label{eq:DeMu05-org}
\end{gather}
Moreover, for the conformal factor $h:S^2\to\R^+$ given by $\psi_\sharp g=h^2\sigma$, $\sigma$ the standard metric on $S^2$, we have by \cite[Theorem 2]{DeMu06}
\begin{gather}
	\|h-1\|_{W^{1,2}(S^2)}+\|h-1\|_{C^0(S^2)}\,\leq\, C\|\Atf\|_{L^2(\Sigma)} \label{eq:DeMu06-org}
\end{gather}
for a universal constant $C>0$. Fixing such a parametrization $\psi$ we define
\begin{gather*}
	N:S^2\,\to\, S^2,\, N\,:=\, \nu\circ\psi.
\end{gather*}
Note that 
\begin{gather*}
	N(x)
	\,=\, \frac{d\psi(x)(\tilde{\tau}_1)\times d\psi(x)(\tilde{\tau}_2)}{|d\psi(x)(\tilde{\tau}_1)\times d\psi(x)(\tilde{\tau}_2)|},
\end{gather*}
where $(\tilde{\tau}_1,\tilde{\tau}_2,x)$ is an orthonormal basis of $\R^n$ in $x\in S^2$.\\
By \eqref{eq:DeMu05-org}, \eqref{eq:DeMu06-org} we deduce
\begin{gather}
	\|\psi-N\|_{W^{1,2}(S^2)}\,\leq\, C\|\Atf\|_{L^2(\Sigma)}. \label{eq:psiN-small-org}
\end{gather}
Around a point $x_0\in\Sigma$, $x_0=\psi(\xi_0)$ we often use a local parametrization of the following type. Denote by $\D_r:= B(0,r)\subset\R^2$ the open ball in $\R^2$ with radius $r>0$ and center $0$. Let $\sgp:S^2 \setminus\{-\xi_0\}\to\R^2$ denote the standard stereographic projection that maps $\xi_0$ to  the origin and the equator $S^2\cap \{\xi_0\}^\perp$ to $\partial\D_1\subset\R^2$. We then define
\begin{align*}
	\Psi\,&:\, \D_1\,\to\, \Sigma, &\Psi\,&:=\, \psi\circ \sgp^{-1},\\
	M\,&:\, \D_1\,\to\, S^2,& M\,&:=\, N\,\circ \sgp^{-1}\,=\, \nu\circ\Psi.
\end{align*}
We deduce from \eqref{eq:DeMu05-org}, \eqref{eq:DeMu06-org}, and \eqref{eq:psiN-small-org} that for $\Sigma\in \M_{4\pi}$ with $\W(\Sigma)\leq 6\pi$
\begin{align}
	\|\Psi -\sgp^{-1}\|_{W^{2,2}(\D_1)}\,&\leq\, C\|\Atf\|_{L^2(\Sigma)}, \label{eq:DeMu05}\\
	\||J\Psi|-|J\sgp^{-1}|\|_{C^0(\D_1)}\,&\leq\, C\|\Atf\|_{L^2(\Sigma)}, \label{eq:DeMu06}\\
	\|\Psi-M\|_{W^{1,2}(\D_1)}\,&\leq\, C\|\Atf\|_{L^2(\Sigma)}. \label{eq:psiN-small}	
\end{align}
Since $1-|\Psi|^2\,=\, (\sgp^{-1} - \Psi)\cdot (\sgp^{-1} + \Psi)$ and since $|\Psi|>\frac{1}{2}$ for $\|\Atf\|_{L^2(\Sigma)}$ sufficiently small this yields
\begin{align}
	\|1-|\Psi|^2\|_{W^{2,2}(\D_1)} + \|1-|\Psi|\|_{W^{2,2}(\D_1)} \,\leq\, C\|\Atf\|_{L^2(\Sigma)} \label{eq:Psi-1-small}
\end{align}
for $\|\Atf\|_{L^2(\Sigma)}$ sufficiently small.

In order to prove Theorem \ref{thm:lb-4pi} we fix $\Sigma_a\in \M_a$ with $a>4\pi$ and define
\begin{gather*}
	\delta\,:=\, \sqrt{\W(\Sigma_a)-4\pi}\,=\, \frac{\sqrt{2}}{2}\|\Atf\|_{L^2(\Sigma_a)}.
\end{gather*}
It is sufficient to prove \eqref{eq:lb-4pi} for all $\delta< \delta_0$, where $\delta_0>0$ is an arbitrary universal constant, since for $\delta\geq\delta_0$ by \eqref{eq:lb-lin}
\begin{gather*}
	\W(\Sigma_a)-4\pi \,\geq\, {\delta_0}\sqrt{a-4\pi}
\end{gather*}
holds. In the following we assume $\delta_0<\sqrt{2\pi}$, associate to $\Sigma_a$ the dilated surface $\Sigma=\sqrt{\frac{4\pi}{a}}\Sigma_a$ with $\ar(\Sigma)=4\pi$,  and let $\lambda={\frac{a}{4\pi}}$. By \cite{DeMu05,DeMu06} there exists a  conformal parametrization $\psi:S^2\to\Sigma$ with \eqref{eq:DeMu05-org}-\eqref{eq:Psi-1-small}. By choosing $\delta_0>0$ sufficiently small we can moreover assume that
\begin{gather}
	\frac{1}{2}\,\leq\, \lambda \,\leq \, {2}, \label{eq:est-pi-a}\\
	|\psi|\,\geq\, \frac{1}{2}, \label{eq:est-Psi1half}\\
	\frac{1}{2}\,\leq |J\Psi|\,\leq 5 \label{eq:est-JPsi}
\end{gather}
for any local parametrization $\Psi:\D_1\to\Sigma$ a above.

To derive the desired lower bound we use \eqref{eq:K-M} for $\Sigma_a$ and estimate the right-hand side of this inequality from above. We observe that
\begin{gather*}
	1-(x\cdot\nu(x))^2 \,=\, 1 - |x|^2 + |x- (x\cdot\nu(x))\nu(x)|^2
\end{gather*}
%the integral can be written as
%\begin{align*}
%	&\int_{\Sigma_a} \big(1-|x|^2\big)K_a(x)\,d\Ha^2(x) + \frac{3}{2}\int_{\Sigma_a} \eta(x) \big|x-(x\cdot\nu(x))\nu(x)\big|^2K_a(x)\,d\Ha^2.
%\end{align*}
and reformulate \eqref{eq:K-M} in terms of the dilated surface $\Sigma$ as
\begin{align}
	&-\big(a-4\pi\big)\notag\\
	=\, & \int_{\Sigma} \big(1-\lambda |x|^2\big)K(x)\,d\Ha^2(x) + \frac{3\lambda}{2}\int_\Sigma \eta(x) \big|x-(x\cdot\nu(x))\nu(x)\big|^2K(x)\,d\Ha^2. \label{eq:K-M-lambda}
\end{align}
We have to show that both terms on the right-hand side are bounded from below by $-C\delta^4$.
\begin{remark-short}
 Let us first briefly outline the intuition behind the proof of these lower bounds.
For the second term on the right-hand side of \eqref{eq:K-M-lambda} the lower bound is easy if one has slightly stronger assumptions than \eqref{eq:DeMu05-org}-\eqref{eq:Psi-1-small}. Indeed since $K=\det A$ we get from \eqref{eq:DeMu05-org} and \eqref{eq:DeMu06-org} that $\LL^2(\{K\leq 0\})\leq C\delta^2$, while \eqref{eq:DeMu05}, \eqref{eq:DeMu06} imply that $\|\nu(x)-x\|_{L^q}\leq C_q\delta$ for all $q<\infty$. If we had an $L^\infty$ bound the lower bound $-C\delta^4$ for the second term would follow immediately. Now $W^{1,2}$ does not embed into $L^\infty$ but into $\BMO$, the space of functions of bounded mean oscillation. This space is dual to the Hardy space $\Hardy^1$ and since the Gauss curvature has the structure of a determinant one might expect that we can bound $K-1$ not only in $L^1$ but in $\Hardy^1$. One can, however, not rely directly on the $\BMO-\Hardy^1$ duality since, e.g., $\BMO$ is not an algebra and $\|f^2\|_{\BMO}$ cannot be estimated by $\|f\|_{\BMO}^2$. Instead, similar to \cite{DeMu05} one has to carefully approximate $|K-1||x-(x\cdot\nu(x))\nu(x)|^2$ in a way which preserves as much of the determinant structure as possible, see in particular \eqref{eq:KdetF}, \eqref{eq:split-2}, and Proposition \ref{prop:6}. For the first integral on the right-hand side of \eqref{eq:K-M-lambda} the estimate $\LL^2(\{K\leq 0\})\leq\, C\delta^2$, \eqref{eq:Psi-1-small}, and the Sobolev embedding $W^{2,2}\hookrightarrow L^\infty$ give immediately the lower bound $-C\delta^3$, but this bound has the wrong exponent $3$ instead of $4$. To get a better bound we exploit that $K\geq \frac{1}{2} - 2|A-\id|^2$ (see below) and that $f(x)=1-\lambda |x|^2$ has small oscillations on small balls. Indeed, if $W^{2,2}$ would embed into $W^{1,\infty}$ we knew that $f$ is Lipschitz with Lipschitz constant $C\delta$, hence that $\osc_{\D_r} f \leq C\delta r$. Now the embedding  from $W^{2,2}$ to $W^{1,\infty}$ again just fails, but we can use Lemma \ref{eq:lemma-g3} below as a substitute.
\end{remark-short}

We now start with the rigorous estimate of the integrals in \eqref{eq:K-M-lambda}. We use a partition of unity on $S^2$ and local parametrizations $\psi$ as described above. %We thus cover $S^2$ by finitely many sets of the form $\sgp^{-1}(\D_1)$ and choose a corresponding partition of unity. 
We then have to estimate expressions of the form
\begin{gather}
	\int_{\D_1} \eta(y) \big(1-\lambda|\Psi(y)|^2\big)K(\Psi(y)) |J\Psi(y)|\,dy \notag\\
	+ \frac{3\lambda}{2}\int_{\D_1} \eta(y) \big|\Psi(y)-(\Psi(y)\cdot M(y)) M(y)\big|^2K(\Psi(y)) |J\Psi(y)|\,dy \label{eq:split-integrand}
\end{gather}
from below, where $\eta$ is a smooth localization,
\begin{align}
	\eta\in C^\infty_c(\D_1),\quad 0\leq \eta\leq 1,\quad \|\eta\|_{C^1(\D_1)}\,\leq\, C. \label{eq:eta-localize}
\end{align}
We proceed in several steps.
\subsection{First term in \eqref{eq:split-integrand}}
In this subsection we prove the following Proposition.
\begin{proposition}\label{prop:est1.1}
There exists $C>0$ such that for any $\eta$ as in \eqref{eq:eta-localize}, $c_0>0$, and all $\delta<\delta_0$ sufficiently small
\begin{align}
	&\int_{\D_1} \eta(y) \big(1-\lambda |\Psi(y)|^2\big)K(\Psi(y)) |J\Psi(y)|\,dy \notag\\
	\geq\,& \frac{1}{2} \int_{\D_1} \eta \big(1-\lambda|\Psi|^2\big)|J\Psi| - \frac{C}{c_0^2} \int_{\D_1} (1-\lambda|\Psi|^2) - Cc_0\delta^4. \label{eq:prop-est1.1}
\end{align}
\end{proposition}
In the remainder of this subsection we prove Proposition \ref{prop:est1.1}. We start by observing that
\begin{align*}
	\tr (A-\id)\,&\leq\, \sqrt{2} |A-\id| \,\leq\, \frac{1}{2} + |A-\id|^2,\\
	\det(A-\id)\,&\leq\, |A-\id|^2,
\end{align*}
which yields
\begin{gather*}
	K\,=\, \det A \,=\, \det (\id+A-\id)\,=\, 1 +\tr (A-\id) + \det (A-\id) \,\geq\, \frac{1}{2} - 2|A-\id|^2.
\end{gather*}
We therefore obtain for the the left-hand side of \eqref{eq:prop-est1.1}
\begin{align}
	&\int_{\D_1} \eta(y) \big(1-\lambda|\Psi(y)|^2\big)K(\Psi(y)) |J\Psi(y)|\,dy \notag\\
	=\,
  	&\frac{1}{2}\int_{\D_1} \eta(y) \big(1-\lambda|\Psi(y)|^2\big)|J\Psi(y)|\,dy \notag\\
	&- 2\int_{\D_1} \eta(y) \big(1-\lambda|\Psi(y)|^2\big)|A(\Psi(y))-\id|^2 |J\Psi(y)|\,dy \label{eq:1st-term1}
\end{align}
Below we will cover $\D_1$ by smaller balls and control the right-hand side by using the positive contribution from the first term and the smallness of $\|A-\id\|_{L^2(\Sigma)}$.
We need the following auxiliary result.
\begin{lemma}\label{eq:lemma-g3}
For any nonnegative $f\in W^{2,2}(\D_r)$, $0<r\leq1$, 
\begin{align}
	\sup_{\D_r}f \,&\leq\, 2\fint_{D_r}f + 2 Cr\|D^2f\|_{L^2(\D_r)} \label{eq:g3-2}
\end{align}
holds.
\end{lemma}
\begin{proof}
Set $a_r:=\fint_{\D_r}f$, $A_r:=\fint_{\D_r}\nabla f$ and define $h(y):=f(y)-a_r -A_r \cdot y$. We first prove
\begin{align}
	\|h\|_{L^\infty(\D_r)} \,&\leq\, Cr\|D^2f\|_{L^2(\D_r)}. \label{eq:g3-1}
\end{align}
Since the estimate is invariant under the rescaling $f_r(y)=f(ry)$ it is sufficient to prove the claim for $r=1$. We obtain by the Poincar{\'e} inequality
\begin{align*}
	\|\nabla h\|_{L^2(\D_1)}\,=\, \|\nabla f - \fint\nabla f \|_{L^2(\D_1)} &\leq\,  C\|D^2f\|_{L^2(\D_1)},\\
	\|h\|_{L^2(\D_1)}\,=\, \| h - \fint h \|_{L^2(\D_1)} &\leq\,  C\|\nabla h\|_{L^2(\D_1)}
\end{align*}
and deduce that $\|h\|_{W^{2,2}(\D_1)}\leq C\|D^2f\|_{L^2(\D_1)}$. By the Sobolev embedding Theorem we deduce \eqref{eq:g3-1}. 

Next we obtain from \eqref{eq:g3-1}
\begin{align*}
	\sup_{\D_r} f \,&=\, \sup_{y\in \D_r} \Big(a_r + A_r\cdot y + h(y)\Big)\,\leq\, a_r+r|A_r|+ \|h\|_{L^\infty(\D_r)}\\
	&\leq\, a_r+r|A_r|+Cr\|D^2f\|_{L^2(\D_r)}
\end{align*}
and
\begin{align*}
	0\leq \inf_{\D_r} f \,&=\, \inf_{y\in \D_r} \Big(a_r + A_r\cdot y + h(y)\Big)\,\leq\, a_r + A_r\cdot \frac{-rA(r)}{|A(r)|} + \sup_{\D_r} |h(y)|\\
	& \leq\, a_r -r|A_r| + Cr\|D^2f\|_{L^2(\D_r)}.
\end{align*}
Combining both  inequalities \eqref{eq:g3-2} follows.
\end{proof}
\begin{proof}[Proof of Proposition \ref{prop:est1.1}]
There exists a universal constant $C_B\in\N$ and a finite partition of unity $1=\sum_{i=1}^N \vartheta_i$ on $\D_1$ such that 
\begin{gather*}
	\nbr\{1\leq i\leq N\,:\, y\in\spt(\vartheta_i)\}\leq C_B \quad\text{ for all }y\in \D_1
\end{gather*}
and such that $0\leq \vartheta_i\leq 1$ for all $i=1,\dots,N$ and $\vartheta_i\in C^\infty(\D_r(y^i))$ for $r=c_0\delta$ as chosen below.

We apply the previous lemma to the function $f:=\big(1-\lambda |\Psi|^2\big)$. By \eqref{eq:DeMu05} 
\begin{gather}
	\|\Psi-\sgp^{-1}\|_{W^{2,2}(\D_1)} \,\leq\,C\delta \label{eq:g3-3}
\end{gather}
holds, we obtain  $f\in W^{2,2}(\D_1)$, and using \eqref{eq:Psi-1-small}
\begin{gather}
	\|D^2 f\|_{L^{2}(\D_1)}\,\leq\, C\delta. \label{eq:g3-4}
\end{gather}
Since $\lambda\leq 2,|\Psi|\leq 1$ we deduce from \eqref{eq:g3-1}
\begin{align}
	\sup_{ \D_r} (1-\lambda|\Psi|^2) \,\leq\, 2 \fint_{\D_r}  (1-\lambda |\Psi|^2) + 2Cr\delta. \label{eq:g3-99}
\end{align}	
This yields for all $r<1$ the estimate
\begin{align*}
	 &\int_{\D_r(y^i)} \eta\vartheta_i(1-\lambda |\Psi|^2)|A\circ\Psi-\id|^2|J\Psi|\\
	\leq &\, 2\Big( \fint_{\D_r(y^i)} (1-\lambda|\Psi|^2) + C\delta r\Big) \int_{\D_r(y^i)} \eta\vartheta_i|A\circ\Psi-\id|^2|J\Psi|\\
	\leq &\, \frac{2}{\pi r^2} \Big(\int_{\D_r(y^i)} (1-\lambda|\Psi|^2)\Big)\delta^2 + 2C\delta r \int_{\D_r(y^i)}\eta \vartheta_i|A\circ\Psi-\id|^2|J\Psi|.
\end{align*}
We deduce from \eqref{eq:1st-term1}
\begin{align*}
	&\int_{\D_1} \eta(y) \big(1-\lambda|\Psi(y)|^2\big)K(\Psi(y)) |J\Psi(y)|\,dy \notag\\
	\geq\, 
  	&\frac{1}{2}\int_{\D_1} \eta\big(1-\lambda |\Psi|^2\big)|J\Psi| - 2\sum_{i=1}^N \int_{\D_1} \eta\vartheta_i \big(1-\lambda |\Psi|^2\big)|A(\Psi)-\id|^2 |J\Psi|\\
	\geq\,& \frac{1}{2} \int_{\D_1} \eta \big(1-\lambda|\Psi|^2\big)|J\Psi| - \frac{C_B \delta^2}{\pi r^2} \Big(\int_{\D_1} (1-\lambda|\Psi|^2)\Big) - 2C\delta r \int_{\D_1} \eta |A\circ\Psi-\id|^2|J\Psi|
\end{align*}
By choosing $r= c_0\delta$ we obtain \eqref{eq:prop-est1.1}.
\end{proof}
\subsection{Second term in \eqref{eq:split-integrand}}
As in this term $\lambda$ only appears as a constant prefactor and since $\frac{1}{2}\leq \lambda\leq 2$ we drop the factor $\lambda$ in the following.
We first show that
\begin{gather*}
	\big(K\circ\Psi -1\big) |J\Psi| \,=\, M\cdot\partial_1 M \times \partial_2 M - M\cdot \partial_1\Psi\times\partial_2\Psi
\end{gather*}
can be well approximated by a term which preserves the determinat structure plus an extra error term which is more regular, i.e., in $L^q$ rather than in $L^1$, $q<2$.

For $\Psi:\D_1\to \Sigma$ as above we set $e_3:=\frac{\Psi}{|\Psi|}$. Then $e_3\in W^{2,2}(\D_1)$ and there exist $e_1,e_2\in W^{2,2}(\D_1)$ such that $(e_1(y),e_2(y),e_3(y))$ is an orthonormal basis of $\R^3$ for all $y\in \D_1$. We then define
\begin{align*}
	F_i\,&:=\, M\cdot e_i,\quad i=1,2,3,\\ F\,&:=\, (F_1,F_2,F_3)^T\in S^2,\\
	F'\,&:=\, (F_1,F_2)^T
\end{align*}
and observe that
\begin{align}
	F_i\,=\, (M-\Psi)\cdot e_i\quad\text{ for }i=1,2, \notag\\
	|\Psi - (\Psi\cdot M)M|^2\,=\, |\Psi|^2|F'|^2. \label{eq:Psi-tan}
\end{align}
By \eqref{eq:psiN-small} we have, using $\|fg\|_{W^{1,2}(\D_1)}\leq C\|f\|_{W^{1,2}(\D_1)}\|g\|_{W^{2,2}(\D_1)}$,
\begin{gather}
	\int_{\D_1} |F'|^2 + \int_{\D_1} |\nabla F'|^2\,\leq\, C\delta^2. \label{eq:F-small}
\end{gather}
Furthermore $(M-\Psi)\cdot e_3\,=\, F_3 -|\Psi|$ and 
\begin{align}
	\partial_i F_3 \,=\, \partial_i \big((M-\Psi)\cdot e_3\big) +\frac{1}{|\Psi|}\Psi\cdot\partial_i\Psi
\end{align}
holds and we obtain from \eqref{eq:psiN-small}, \eqref{eq:Psi-1-small} that
\begin{gather}
	\int_{\D_1} |\partial_i F_3|^2 \,\leq\, C\delta^2. \label{eq:F3-small}
\end{gather}
We further compute
\begin{gather}
	\partial_i (M-\Psi)\,=\, \sum_{j=1}^3 (\partial_i F_j) e_j +R_i^{(1)},\quad R_i^{(1)}\,:=\, \sum_{j=1}^3 F_j\partial_i e_j -\partial_i\Psi \label{eq:del-m-psi}
\end{gather}
and claim that
\begin{gather}
	\|R_i^{(1)}\|_{W^{1,p}(\D_1)}\,\leq\, C_p\delta,\quad i=1,2\quad\text{ for all }1\leq p<2. \label{eq:Ri-small}
\end{gather}
In fact,
\begin{gather*}
	\sum_{j=1}^3 F_j\partial_i e_j -\partial_i\Psi\,=\,  F_1\partial_i e_1 + F_2\partial_i e_2 +\big( F_3\partial_i\frac{\Psi}{|\Psi|} - \partial_i\Psi\big).
\end{gather*}
The estimate for the first two terms on the right-hand side follows from \eqref{eq:F-small} and the embedding $W^{1,2}(\D_1)\hookrightarrow L^q(\D_1)$ for all $1\leq q<\infty$, whereas the third term can first be written as
\begin{align*}
	F_3\partial_i\frac{\Psi}{|\Psi|} - \partial_i\Psi\,&=\, \frac{1}{|\Psi|}(F_3-|\Psi|)\partial_i\Psi - \frac{F_3}{|\Psi|^3}\Psi\cdot\partial_i\Psi\\
		\,&=\, \frac{1}{|\Psi|}(M-\Psi)\cdot e_3\partial_i\Psi +\frac{F_3}{2|\Psi|^3}\partial_i(1-|\Psi|^2).
\end{align*}
The estimate then follows from \eqref{eq:psiN-small}, \eqref{eq:Psi-1-small} and the embedding $W^{1,2}(\D_1)\hookrightarrow L^q(\D_1)$.

We next write using \eqref{eq:del-m-psi}
\begin{align}
	M\cdot \partial_1 (M-\Psi) \times \partial_2 (M-\Psi) \,&=\, M\cdot \sum_{j=1}^3 (\partial_1 F_j)e_j \times \sum_{k=1}^3 (\partial_2 F_k)e_k + M\cdot R^{(1)} \notag\\
	\,&=\, F\cdot\partial_1 F \times \partial_2 F + M\cdot R^{(1)},\label{eq:41-1}
\end{align}
with
\begin{align*}
	R^{(1)}\,:=\, &\Big(\sum_{j=1}^3 (\partial_1 F_j) e_j \times R_2^{(1)}\Big) + \Big(R_1^{(1)}\times \sum_{j=1}^3 (\partial_2 F_j) e_j \Big)+ \Big(R_1^{(1)}\times R_2^{(1)}\Big).
\end{align*}
The estimates \eqref{eq:F-small}, \eqref{eq:F3-small}, and \eqref{eq:Ri-small} imply that for all $1\leq q<2$ there exists $C_q>0$ such that
\begin{gather}
	\|R^{(1)}\|_{L^q(\D_1)}\,\leq\, C_q \delta^2. \label{eq:R1-small}
\end{gather}
Furthermore we observe that $M\cdot \partial_1\Psi\times\partial_2\Psi \,=\, |\partial_1\Psi\times\partial_2\Psi|\,=\, |J\Psi|$ and thus
\begin{align}
	&K\circ\Psi |J\Psi| \notag\\
	=\, &M\cdot \partial_1 M \times\partial_2 M \notag\\
	=\, &M\cdot \partial_1(M-\Psi)\times \partial_2 (M-\Psi) \notag\\
	&+ M\cdot \big(\partial_1\Psi\times \partial_2 (M -\Psi) +\partial_1(M-\Psi)\times\partial_2\Psi\big) + |J\Psi| \notag\\
	=\, &F\cdot\partial_1 F \times \partial_2 F +R + |J\Psi|, \label{eq:KdetF}
\end{align}
where by \eqref{eq:41-1}
\begin{gather}
	R\,:=\,  M\cdot R^{(1)}+ M\cdot \big(\partial_1\Psi\times \partial_2 (M -\Psi) +\partial_1(M-\Psi)\times\partial_2\Psi\big). \label{eq:def-R}
\end{gather}
%We further compute that
%\begin{gather}
%	\big|\Psi(y)-(\Psi(y)\cdot M(y)) M(y)\big|^2 \,=\, |\Psi|^2|F'|^2.
%\end{gather}
The main point is that $F$ has values in $S^2$ and $F\cdot \partial_1F\times\partial_2 F$ is just the pull-back of the volume form on $S^2$, so that $F\cdot \partial_1F\times\partial_2 F$ is essentially a two-dimensional determinant (see \eqref{eq:64bis}). If instead we directly expand $M\cdot\partial_1 M\times\partial_2 M$ by setting $M=\Psi+(M-\Psi)$ we get a term $(M-\Psi)\cdot\partial_1 (M-\Psi)\times\partial_2 (M-\Psi)$ which has no such interpretation.

For the following calculations it is convenient to treat the cases $|F_3|$ small and $|F_3|$ close to one differently. We therefore introduce a cut-off function $\ctf$ acting on the values of $F_3^2$, 
\begin{align}
	\ctf \in C^\infty_c(0,1],\quad 0\leq \ctf\leq 1,\quad \ctf|_{[\frac{3}{4},1]}=1,\quad \ctf|_{[0,\frac{1}{2}]}=0. \label{eq:ctf-theta}
\end{align}
Using \eqref{eq:Psi-tan} we then rewrite the second term in \eqref{eq:split-integrand} as
\begin{align}
	&\int_{\D_1} \eta \big|\Psi(y)-(\Psi(y)\cdot M(y)) M(y)\big|^2 K\circ \Psi |J\Psi| \notag\\
	=\, &\int_{\D_1} \eta |\Psi|^2 |F'|^2 \,\big(1-\ctf(F_3^2)\big) F\cdot\partial_1 F \times \partial_2 F \notag\\
	&+ \int_{\D_1}\eta |\Psi|^2 |F'|^2  \ctf(F_3^2) F\cdot\partial_1 F \times \partial_2 F+ \int_{\D_1} \eta |\Psi|^2 |F'|^2 \big( R + |J\Psi|\big) . \label{eq:split-2}
\end{align}
We treat the three terms on the right-hand side separately.
\subsubsection{First term on the right-hand side of \eqref{eq:split-2}}
\begin{proposition}\label{prop:degree}
For $\delta_0>0$ sufficiently small we have
\begin{gather}
	\int_{\D_1} \eta |\Psi|^2 |F'|^2 \big(1-\ctf(F_3^2)\big) F\cdot\partial_1 F \times \partial_2 F \,\geq\, -C\delta^4.
\end{gather}
\end{proposition}
As $\eta$ is compactly supported in $\D_1$ we may extend $\Psi$ to a $W^{2,2}$-map $\Psi:\R^2\to \R^3$. We further consider the square $Q=[-1,1]^2$. For $k\in\N$ fixed it follows from \eqref{eq:F-small} that 
\begin{align*}
	\sum_{j=-k}^k \int_0^{\frac{1}{k}}\int_{-1}^1 |\nabla F'|^2(y_1+\frac{j}{k},y_2)\,dy_2\,dy_1 \,&\leq\, {C}{}\delta^2, \\ 
	\sum_{j=-k}^k \int_0^{\frac{1}{k}}\int_{-1}^1 |\nabla F'|^2(y_1,y_2+\frac{j}{k})\,dy_1 dy_2\,&\leq\, {C}{}\delta^2. 
\end{align*}
Therefore we can choose $a\in [0,\frac{1}{k}]^2$ such that
\begin{align}
	\sum_{j=-k}^k \int_{-1}^1 |\nabla F'|^2(a_1+\frac{j}{k},y_2)\,dy_2 \,&\leq\, {C}{k}\delta^2, \label{eq:a1}\\ 
	\sum_{j=-k}^k \int_{-1}^1 |\nabla F'|^2(y_1,a_2+\frac{j}{k})\,dy_1 \,&\leq\, {C}{k}\delta^2. \label{eq:a2}
\end{align}
Let $Q_j$, $j\in\N$ denote an enumeration of the squares with edge length $\frac{1}{k}$ and corners in the set $\{a+\frac{1}{k}\Z^2\}$ such that $\spt (\eta)\subset \bigcup_{j=1,\dots,N} Q_j$, $N\leq 5k^2$. By \eqref{eq:a1}, \eqref{eq:a2} we have 
\begin{gather}
	\int_{\partial Q_j} |\nabla F'|^2 \,\leq\, Ck\delta^2\quad\text{ for all }j=1,\dots, N. \label{eq:a12}
\end{gather}
In particular, for $\delta_0>0$ small enough we estimate
\begin{align}
	\osc_{\partial Q_j} F' \,&\leq\, C|\partial Q_j|^{\frac{1}{2}} \Big(\int_{\partial Q_j} |\nabla F'|^2\Big)^\frac{1}{2} \notag\\
	&\leq\, C \frac{1}{\sqrt{k}}\sqrt{k}\delta\,\leq\, C\delta. \label{eq:osc}
\end{align}
Furthermore we obtain from \eqref{eq:F-small} that $\int_{Q_j} |F'|^p \leq C_p \delta^p$ for all $1\leq p<\infty$ and
\begin{gather}
	|\big\{F_3^2\leq \frac{5}{6}\big\}\cap Q_j|\,=\, |\big\{|F'|^2 \geq \frac{1}{6}\big\}\cap Q_j|\,\leq\, C_p \int_{Q_j} |F'|^p\,\leq\, C_p\delta^p. \label{eq:set-F-small}
\end{gather}
\begin{lemma}\label{lem:bdryQ}
There exist $\delta_0>0$ and constants $\bar{C}_p>0$, $1\leq p<\infty$ such that for any $0<\delta<\delta_0$, $k\in\N$, and $1\leq p<\infty$ with $\frac{1}{k^2}\geq \bar{C}_p\delta^p$ the inequality
\begin{gather}
	F_3^2\,>\, \frac{3}{4}\quad\text{ on }\partial Q_j \label{eq:lembdryQ}
\end{gather}
holds.
\end{lemma}
\begin{proof}
Assume that $F_3^2(y)=1-|F'|^2(y)\leq \frac{3}{4}$ for a $y\in \partial Q_j$, $j\in \{1,\dots,N\}$. Then we deduce from \eqref{eq:osc} that $|F'|^2\geq \frac{1}{5}$ on $\partial Q_j$ for $\delta_0>0$ small enough. By the Poincar{\'e} inequality on the unit cube, a rescaling argument, and $|F'|\leq 1$ this implies
\begin{align}
	\int_{Q_j} (\frac{1}{5}-|F'|^2)_+^2 \,&\leq\, C\frac{1}{k^2}\int_{Q_j} | D\big(\frac{1}{5}-|F'|^2\big)|^2 \notag\\
	&\leq\, C \frac{1}{k^2} \int_{Q_j} |D F'|^2 \,\leq\, C\frac{\delta^2}{k^2}\,\leq\, \frac{1}{1800}|Q_j| \label{eq:intFsmall}
\end{align}
for $\delta_0>0$ sufficiently small.
Therefore
\begin{gather*}
	|\{|F'|^2 < \frac{1}{6}\}\cap Q_j|\,\leq\, 900 \int_{Q_j} (\frac{1}{5}-|F'|^2)_+^2 \,\leq\,  \frac{1}{2}|Q_j|
\end{gather*}
and in particular by \eqref{eq:set-F-small}
\begin{gather}
	C_p\delta^p \,\geq\, |\{|F'|^2 \geq \frac{1}{6}\}\cap Q_j|\,\geq\, \frac{1}{2}|Q_j|\,=\, \frac{1}{2k^2}. \label{eq:Fbig}
\end{gather}
This gives a contradiction if $\frac{1}{k^2}\geq \bar{C}_p\delta^p$ and if $\bar{C}_p$ is chosen large enough.
\end{proof}
\begin{proof}[Proof of Proposition \ref{prop:degree}]
Let us assume \eqref{eq:lembdryQ}. This implies that the degree $d:=\deg(F,Q_j,\cdot)$ is constant on $\{\xi\in S^2\,:\, \xi_3^2<\frac{3}{4}\}$. If $d\neq 0$ then $\{\xi\in S^2\,:\,\xi_3^2<\frac{3}{4}\}\subset F(Q_j)$ and thus%, denoting by $\sigma$ the volume form on $S^2$,
\begin{align*}
	\Ha^2\big(\{\xi\in S^2\,:\,\xi_3^2< \frac{3}{4}\}\big)\,&\leq\, \int_{F(Q_j)} 1\,d\Ha^2\\
	&\leq\, \int_{Q_j} (\det DF^T DF)^\frac{1}{2} \,\leq\, \int_{Q_j} |D F|^2 \,\leq\, C\delta^2
\end{align*}
by \eqref{eq:F-small}, \eqref{eq:F3-small}. For $\delta<\delta_0$ small enough we therefore obtain a contradiction. This shows that $\deg(F,Q_j,\cdot)=0$ on $\{\xi\in S^2\,:\, \xi_3^2<\frac{3}{4}\}$. Since $\ctf(F_3^2)=1$ on $\{\xi\in S^2\,:\, \xi_3^2\leq \frac{3}{4}\}$ this implies for $g:S^2\to\R$, $g(\xi)=(1-\ctf(\xi_3^2))(\xi_1^2+\xi_2^2)$ and the volume form $\sigma$ on $S^2$ 
that
\begin{gather}
	0\,=\, \int_{Q_j} F^*\big(g\sigma\big)\,=\,\int_{Q_j} (1-\ctf(F_3^2)) |F'|^2  F\cdot \partial_1F \times \partial_2F. \label{eq:deg-g-0}
\end{gather}
We further deduce that
\begin{align}
	&\Big|\int_{Q_j} \eta |\Psi|^2 (1-\ctf(F_3^2))  |F'|^2 F\cdot \partial_1F \times \partial_2F\Big| \notag\\
	\,\leq\, &\Big|\int_{Q_j} \big(\eta |\Psi|^2- (\eta|\Psi|^2)(a^{(i)})\big)(1-\ctf(F_3^2))|F'|^2 F\cdot \partial_1F \times \partial_2F \Big|\notag\\
	&+\Big|(\eta|\Psi|^2)(a^{(i)})\big)\int_{Q_j} F^*\big(g\sigma\big)\Big| \notag\\
	\leq\,& {C_\alpha}(1+\Lip(\eta)){k^{-\alpha}}\int_{Q_j} |D F|^2, \label{eq:DF-alpha}
\end{align}
for $\alpha\in (0,1)$, since $\|1-|\Psi|^2\|_{C^{0,\alpha}(\D_1)}\leq C_\alpha\delta$ by \eqref{eq:Psi-1-small} and since $\int_{Q_j} F^*\big(g\sigma\big)=0$ by  \eqref{eq:deg-g-0}.

We then choose $\alpha=\frac{1}{2}$, $p=8$ and $\delta_0>0$ such that  $C_8\delta_0^8<1$ for the constant $C_8$ from Lemma \ref{lem:bdryQ}.
For $\delta<\delta_0$ we set $k= \lfloor C_8^{-\frac{1}{2}}\delta^{-4} \rfloor$. 
Then \eqref{eq:lembdryQ} is satisfied and \eqref{eq:DF-alpha} shows
\begin{align*}
	\int_{Q_j} \eta |\Psi|^2 |F'|^2 (1-\ctf(F_3^2)) F\cdot \partial_1F \times \partial_2F  &\geq\, -C\delta^2(1+\Lip(\eta))\int_{Q_j} |D F|^2,
\end{align*}
and by \eqref{eq:F-small} and \eqref{eq:F3-small} the claim follows.
\end{proof}
\subsubsection{Second term on the right-hand side of \eqref{eq:split-2}}
Since $\ctf\in C^\infty_c((\frac{1}{2},1])$ we can represent $\D_1\cap \spt(\ctf\circ F_3^2)$ as the disjoint union of the sets
\begin{align*}
	\D_1^+ \,:=\, \D_1\cap \{F_3>\frac{1}{2}\sqrt{2}\} \quad\text{ and }\quad\D_1^- \,:=\, \D_1\cap \{-F_3>\frac{1}{2}\sqrt{2}\}.
\end{align*}
We then have
\begin{gather*}
	F_3\,=\, \pm \sqrt{1-|F'|^2} \text{ on } \D_1^\pm
\end{gather*}
and
\begin{gather*}
	\partial_i F_3\,=\, -\frac{1}{F_3} F'\cdot\partial_i F',\quad\text{ i=1,2.}
\end{gather*}
A short computation shows that
\begin{gather}
	F\cdot \partial_1 F\times \partial_2 F \,=\, \frac{1}{F_3}\det F'\quad\text{ on } \D_1^+\cup \D_1^-. \label{eq:64bis}
\end{gather}
In the following we only consider the set $\D_1^+$, the set $\D_1^-$ can be treated in the same way. We write
\begin{gather*}
	\int_{\D_1^+}\eta |\Psi|^2 |F'|^2  \ctf(F_3^2) F\cdot\partial_1 F \times \partial_2 F\,=\, \int_{\D_1^+} \eta |\Psi|^2  \ctf(1-|F'|^2)\frac{|F'|^2}{\sqrt{1-|F'|^2}}\det DF'. \end{gather*}
\begin{proposition}\label{prop:6}
For any $\eta\in C^\infty_c(\D_1)$ the estimate
\begin{align*}
	\Big| \int_{\D_1^+}  \eta |\Psi|^2 \ctf(F_3^2)\frac{|F'|^2}{\sqrt{1-|F'|^2}} \det DF'\Big| \,\leq\,  C \delta^4
\end{align*}
holds.
\end{proposition}
\begin{proof}
To rewrite the integrand we use that  for any differentiable $h:\R^2\to\R^2$ 
\begin{gather*}
	\nabla\cdot \Big(\cof DF'^T h(F')\Big)\,=\, (\nabla\cdot  h)(F') \det DF'
\end{gather*}
holds and construct  $h\in C^\infty(\R^2,\R^2)$ with
\begin{align}
	\nabla\cdot  h(z) \,=\, \frac{|z|^2\ctf(1-|z|^2)}{\sqrt{1-|z|^2}}, \notag\\
	|h(z)| \,\leq\,  C |z|^3 \quad\text{ for all }z\in \D_1. \label{eq:h-cube}
\end{align}
Extend $\ctf$ to $\ctf\in C^\infty_c((-\infty,2])$ by setting $\ctf(t)=0$ for $t\leq 0$, $\ctf(t)=1$ for $t\geq 1$. Then $\ctf(1-|z|^2)=0$ for $|z|^2\geq \frac{1}{2}$ and thus there exists a solution $q\in C^\infty(\R^2)$ of
\begin{gather*}
	\Delta q\,=\, \frac{|z|^2\ctf(1-|z|^2)}{\sqrt{1-|z|^2}},
\end{gather*}
which satisfies
\begin{gather*}
	\limsup_{z\to\infty} \frac{|q(z)|}{\ln (z)}\,<\,\infty.
\end{gather*}
Let  $T_3q$ denote the third order Taylor approximation of $q$ in $z=0$. We define $\tilde{q}(z):= q(z) - (T_3q)(z)$ and set $h(z)=\nabla \tilde{q}(z)$. Then all derivatives  of $h$  in $z=0$ up to second order vanish and $h(z)\leq C |z|^3$ holds for a suitable constant $C>0$. This is clear for $|z|\leq R$; on the other hand $q$ is harmonic on $\R^3\setminus B(0,R)$ and grows at most logarithmically. Hence $\nabla q$ is harmonic and satisfies $|\nabla q(z)|\leq \frac{C}{|z|}$ as $z\to\infty$.

Furthermore we deduce that
\begin{gather*}
	\nabla\cdot h\,=\, \Delta \tilde{q} \,=\, \frac{|z|^2\ctf(1-|z|^2)}{\sqrt{1-|z|^2}}.
\end{gather*}
As $\eta (\ctf\circ F_3^2)$ is compactly supported in $\D_1^+$ this implies
\begin{align*}
	 \int_{\D_1^+}  \eta |\Psi|^2 \ctf(F_3^2)\frac{|F'|^2}{\sqrt{1-|F'|^2}} \det DF' \,&=\,  \int_{\D_1^+}  \eta|\Psi|^2 \nabla\cdot  \Big(\cof DF'^T h(F')\Big)\\
	 &=\, \int_{\D_1^+}  \nabla\big(\eta|\Psi|^2\big)\cdot \Big(\cof DF'^T h(F')\Big).
\end{align*}
The integral on the right-hand side is estimated by
\begin{align*}
	&\Big|\int_{\D_1^+}  \nabla\big(\eta|\Psi|^2\big)\cdot \Big(\cof DF'^T h(F')\Big) \Big| \\
	\leq\,& \|\nabla (\eta|\Psi|^2)\|_{L^8(\D_1)}\|h(F')\|_{L^{8/3}(\D_1)}\|DF'\|_{L^2(\D_1)}\\
	\leq\, &C(1+\|\nabla\eta\|_{C^0(\D_1)})(1+\|\nabla\Psi\|_{L^8(\D_1)})\|F'\|_{L^{8}(\D_1)}^3\|DF'\|_{L^2(\D_1)}\\
	\leq\, &C(1+\|\nabla\eta\|_{C^0(\D_1)})\|DF'\|_{W^{1,2}(\D_1)}^4\,\leq\, C\|\eta\|_{C^1(\D_1)}\delta^4,
\end{align*}
where we have used \eqref{eq:h-cube}, the Sobolev inequality, and \eqref{eq:F-small}.
\end{proof}
\subsubsection{Third term on the right-hand side of \eqref{eq:split-2}}
\begin{proposition}\label{prop:junk}
For any $c_1>0$ and any $\delta<\delta_0$ we have
\begin{align} \label{eq:66}
	\int_{\D_1} \eta |\Psi|^2 |F'|^2 \big( R + |J\Psi|\big)\,\geq\,  -C(1+c_1)\delta^4 + \int_{\D_1} \eta  |F'|^2|J\Psi|  - \frac{C}{c_1}  \int_{\D_1 } |F'|^2.
\end{align}
\end{proposition}
\begin{proof}
We recall from \eqref{eq:R1-small},\eqref{eq:def-R} that
\begin{gather*}
	R\,=\,  M\cdot R^{(1)}+ M\cdot \big(\partial_1\Psi\times \partial_2 (M -\Psi) +\partial_1(M-\Psi)\times\partial_2\Psi\big),\\
	\|R^{(1)}\|_{L^q(\D_1)}\,\leq\, C_q \delta^2 \quad\text{ for any }1\leq q<2.
\end{gather*}
Together with \eqref{eq:F-small} the last estimate implies
\begin{align}
	\Big| \int_{\D_1} \eta |\Psi|^2 |F'|^2 M\cdot R^{(1)}\Big| \,&\leq\, C\|F'\|_{L^6(\D_1)}^2 \|R^{(1)}\|_{L^{\frac{3}{2}}(\D_1)} \notag\\
	&\leq\, C\|F'\|_{W^{1,2}(\D_1)}^2\delta^2\,\leq\, C\delta^4. \label{eq:66bis}
\end{align}
We moreover observe that
\begin{gather*}
	M\cdot \big(\partial_1\Psi\times \partial_2 (M -\Psi) +\partial_1(M-\Psi)\times\partial_2\Psi\big) \,=\, \big(H\circ\Psi -2\big) |J\Psi|.
\end{gather*}
It remains to show that
\begin{gather}
	\int_{\D_1} \Big(\eta  |F'|^2 |\Psi|^2  \big( H\circ\Psi - 1\big) -  \eta  |F'|^2 \Big)|J\Psi| \,\geq\,  -Cc_1\delta^4    - \frac{C}{c_1}  \int_{\D_1 } |F'|^2.
	 \label{eq:R+Jac}
\end{gather}
We proceed similarly as in the proof of Proposition \ref{prop:est1.1}. Choose a finite partition of unity $1=\sum_{i=1}^N \vartheta_i$ on $\D_1$ such that  $\nbr\{1\leq i\leq N\,:\, y\in\spt(\vartheta_i)\}\leq C_B$ for all $y\in \D_1$ and such that $0\leq \vartheta_i\leq 1$ for all $i=1,\dots,N$ and $\vartheta_i\in C^\infty(\D_r(y^i))$ for $r=c_1\delta$ chosen below. 
We prove the following auxiliary result.
\begin{lemma}
Let $r>0$ and $f\in W^{1,2}(\D_r,\R^2)$, $h\in L^2(\D_r)$. Then
\begin{gather}
	\Big| \int_{\D_r} |f|^2h \Big| \,\leq\, Cr \|D f\|_{L^2(\D_r)}^2\|h\|_{L^2(\D_r)} + \Big(\int_{\D_r} |f|^2\Big)\frac{C}{r}\|h\|_{L^2(\D_r)}. \label{eq:Lady}
\end{gather}
\end{lemma}
\begin{proof}
This is proved like the Ladyzhenskaya estimate $\|g\|_{L^4(\R^2)}\,\leq\, C \|g\|_{L^2(\R^2)}\|Dg\|_{L^2(\R^2)}$ \cite{Lady59}. Indeed, first observe that the desired estimate is invariant under dilation and it hence suffices to consider $r=1$. Now
\begin{align*}
	\|D |f|^2\|_{L^1(\D_1)} \,&=\, \|2f Df\|_{L^1(\D_1)} \,\leq\, 2\|f\|_{L^2(\D_1)}\|Df\|_{L^2(\D_1)} \notag\\
	&\leq\, \|f\|_{L^2(\D_1)}^2 + \|Df\|_{L^2(\D_1)}^2. 
\end{align*}
Since $\| |f|^2 \|_{L^1(\D_1)}=\|f\|_{L^2(\D_1)}^2$ the Sobolev embedding $W^{1,1}(\D_1)\hookrightarrow L^2(\D_1)$ yields
\begin{align*}
	\| |f|^2 \|_{L^2(\D_1)} \,\leq\, C\Big(\|f\|_{L^2(\D_1)}^2 + \|Df\|_{L^2(\D_1)}^2\Big).
\end{align*}
This implies 
\begin{gather*}
	\Big| \int_{\D_r} |f|^2h \Big| \,\leq\, \| |f|^2\|_{L^2(\D_1)} \|h\|_{L^2(\D_1)} \,\leq\,  C\Big(\|f\|_{L^2(\D_1)}^2 + \|Df\|_{L^2(\D_1)}^2\Big)\|h\|_{L^2(\D_1)},
\end{gather*}
which yields \eqref{eq:Lady} for $r=1$ and hence for all $r>0$. 
\end{proof}
%
\begin{comment}
\begin{lemma}
Let $r>0$ and $f\in W^{1,2}(\D_r)$. Then
\begin{align}
	\| f^2 \|_{L^2(\D_r)} \,\leq\, Cr \|Df\|_{L^2(\D_r)}^2+\frac{C}{r}\|f\|_{L^2(\D_r)}^2. \label{eq:Lady}
\end{align}
\end{lemma}
\begin{proof}
This is proved like the Ladyzhenskaya estimate $\|g\|_{L^4(\R^2)}\,\leq\, C \|g\|_{L^2(\R^2)}\|Dg\|_{L^2(\R^2)}$ \cite{Lady59}. Indeed, first observe that the desired estimate is invariant under dilation and it hence suffices to consider $r=1$. Now
\begin{align*}
	\|D f^2\|_{L^1(\D_1)} \,&=\, \|2f Df\|_{L^1(\D_1)} \,\leq\, 2\|f\|_{L^2(\D_1)}\|Df\|_{L^2(\D_1)} \notag\\
	&\leq\, \|f\|_{L^2(\D_1)}^2 + \|Df\|_{L^2(\D_1)}^2. 
\end{align*}
Since $\| f^2 \|_{L^1(\D_1)}=\|f\|_{L^2(\D_1)}^2$ the Sobolev embedding $W^{1,1}(\D_1)\hookrightarrow L^2(\D_1)$ yields
\begin{align*}
	\| f^2 \|_{L^2(\D_1)} \,\leq\, C\Big(\|f\|_{L^2(\D_1)}^2 + \|Df\|_{L^2(\D_1)}^2\Big).
\end{align*}
This implies \eqref{eq:Lady} for $r=1$ and hence for all $r>0$. 
\end{proof}
\end{comment}
%
Fix $1\leq i\leq N$ and apply the previous lemma for $f=F'$ on $\D_r(y^i)$. 
Note that by \eqref{eq:DeMu05-org}
\begin{gather*}
	\int_{\D_1} \big(H\circ\Psi -2\big)^2\,|J\Psi| \,\leq\, C\delta^2.
\end{gather*}
Using that $|\Psi|\leq 1$, $0\leq\eta\vartheta_i\leq 1$, and \eqref{eq:est-JPsi} we then obtain
\begin{align}
	&\Big| \int_{D_r(y^i)} \eta \vartheta_i |F'|^2 |\Psi|^2 (H\circ\Psi-2) |J\Psi| \Big| \notag\\
	\leq\,& Cr \|DF'\|_{L^2(\D_r(y^i))}^2\|H\circ\Psi-2\|_{L^2(\D_r(y^i))} 
	+ \frac{C}{r}\|H\circ\Psi-2\|_{L^2(\D_r(y^i))}\int_{\D_r} |F'|^2 \notag\\	
	\leq\,& Cr\delta \|DF'\|_{L^2(\D_r(y^i))}^2 + \frac{C}{r}\delta \int_{\D_r(y^i)} |F'|^2. \label{eq:63}
\end{align}
Similarly we have
\begin{align}
	&\int_{\D_r(y^i)} \eta\vartheta_i |F'|^2\big(|\Psi|^2-1\big)|J\Psi| \notag\\
	\leq\, &\Big(Cr \|DF'\|_{L^2(\D_r(y^i))}^2+\frac{C}{r}\int_{\D_r(y^i)}|F'|^2\Big)\||\Psi|^2-1\|_{L^2(\D_r(y^i))} \notag\\
	\leq\, &\Big(Cr \|DF'\|_{L^2(\D_r(y^i))}^2+\frac{C}{r}\int_{\D_r(y^i)}|F'|^2\Big)C\delta \label{eq:69bis}
\end{align}
by \eqref{eq:Psi-1-small}. Summing \eqref{eq:63}, \eqref{eq:69bis} over $i$ we we get with $\|DF'||_{L^2(\D_1)}\leq C\delta$
\begin{align}
	&\int_{\D_1} \eta|F'|^2|\Psi|^2\big(H\circ\Psi-2\big)|J\Psi| + \eta|F'|^2\big(|\Psi|^2-1\big)|J\Psi| \notag\\
	\geq\, & -CC_Br\delta^3 -\frac{CC_B\delta}{r}\int_{\D_1}|F'|^2. \label{eq:73}
\end{align}
Now let $r = c_1 \delta$. Then \eqref{eq:73} implies \eqref{eq:R+Jac}.
\end{proof}
\subsection{Conclusion}
We are now ready to prove Theorem \ref{thm:lb-4pi}. We choose a partition of unity $1=\sum_{i=1}^6 \tilde{\eta}_i$ on $S^2$ such that for each $i=1,\dots,6$ the function $\tilde{\eta}_i$ are given as $\eta_i\circ \sgp_i^{-1}$ where $\eta_i \in C^\infty_c(\D_1)$ and $\sgp_i$ is a standard stereographic projection. From Proposition \ref{prop:est1.1}, Proposition \ref{prop:degree}, Proposition \ref{prop:6}, and Proposition \ref{prop:junk} we obtain that there exists $\delta_0>0$ such that for all $c_0,c_1>0$ and any $\delta<\delta_0$  
\begin{align*}
	&\int_\Sigma \Big(1-\lambda(x\cdot\nu(x))^2 +
  	\frac{\lambda}{2}|x- (x\cdot\nu(x))\nu(x)|^2\Big) K(x)\,d\Ha^2(x) \\
	\geq\, &-\sum_{i=1}^6  C\big(1+c_0 + c_1 \big) \delta^4 + 
	\sum_{i=1}^6  \Big(\frac{1}{2} \int_{\Sigma} \tilde{\eta}_i \big(1-\lambda|x|^2\big) - \frac{C}{c_0^2} \int_{\Sigma} (1-\lambda|x|^2)\Big)\\
	&  +\sum_{i=1}^6  \Big(\int_{\Sigma} \tilde{\eta}_i(x) |x-(x	\cdot\nu(x))\nu(x)|^2  - \frac{C}{c_1}  \int_{\D_1 } |x-(x\cdot\nu(x))\nu(x)|^2\Big)\\
  	\geq\,& -C(1+c_0+c_1)\delta^4 +\Big(\frac{1}{2}- \frac{C}{c_0^2}\Big) \int_{\Sigma} (1-\lambda|x|^2) +  \Big(1-\frac{C}{c_1}\Big)\int_{\Sigma}  |x-(x	\cdot\nu(x))\nu(x)|^2 .
\end{align*}
Choosing $c_0,c_1$ large enough the last two terms become nonnegative. Together with \eqref{eq:K-M-lambda} this proves
\begin{gather*}
	a-4\pi \,\leq\, C\delta^4 \,=\, C (\W(\Sigma)-4\pi)^2
\end{gather*}
for all $\delta<\delta_0$ and all $\Sigma\in\M_a$. This concludes the proof of Theorem \ref{thm:lb-4pi}.
%=====================================================
% proof Thm2
%=====================================================
\section{Proof of Theorem \ref{thm:m-sphere}}
\label{sec:proof}
We have to construct a sequence
$(\Sigma_j)_{j\in\N}\subset \M_{8\pi}$ such that $\W(\Sigma_j)\to 8\pi$ as
$j\to \infty$. 
\step{1}
Depending on a parameter $0<r<1$ we construct a curve $\gamma_+$ in
the upper right quarter of the $(x,y)$-plane and obtain a surface
$\Sigma_+$ in space by rotating $\gamma_+$ around the $y$-axis.

For $0<r<1$ given we determine
$0<r_1<1$, $(x_1,y_1), (x_0,y_0)\in B_1(0)$, $0<\beta<\pi/2$, and
$0<\lambda<x_0$ 
such that (see Figure \ref{fig:const})
\begin{itemize}
\item
the sphere $S_{r_1}(x_1,y_1)$ touches the unit sphere from
inside at $(\cos(\beta),\sin(\beta))$,
\item
the catenoid $\{(x,y): y = y_0 \pm \lambda \arccosh(\frac{x}{\lambda})\}$
touches $S_{r_1}(x_1,x_2)$ for $x=\lambda$ in
$(x_1,y_1)+r_1(\cos(\pi/2+\beta),\sin(\pi/2+\beta))$,
\item
the catenoid $\{(x,y): y = y_0 \pm \lambda \arccosh(\frac{x}{\lambda})\}$
touches $S_r(0)$ for $x=-\lambda$ in $(\cos(\pi/2-\beta),\sin(\pi/2-\beta))$.
\end{itemize}
\begin{figure}[h]
  \includegraphics*[width=12cm]{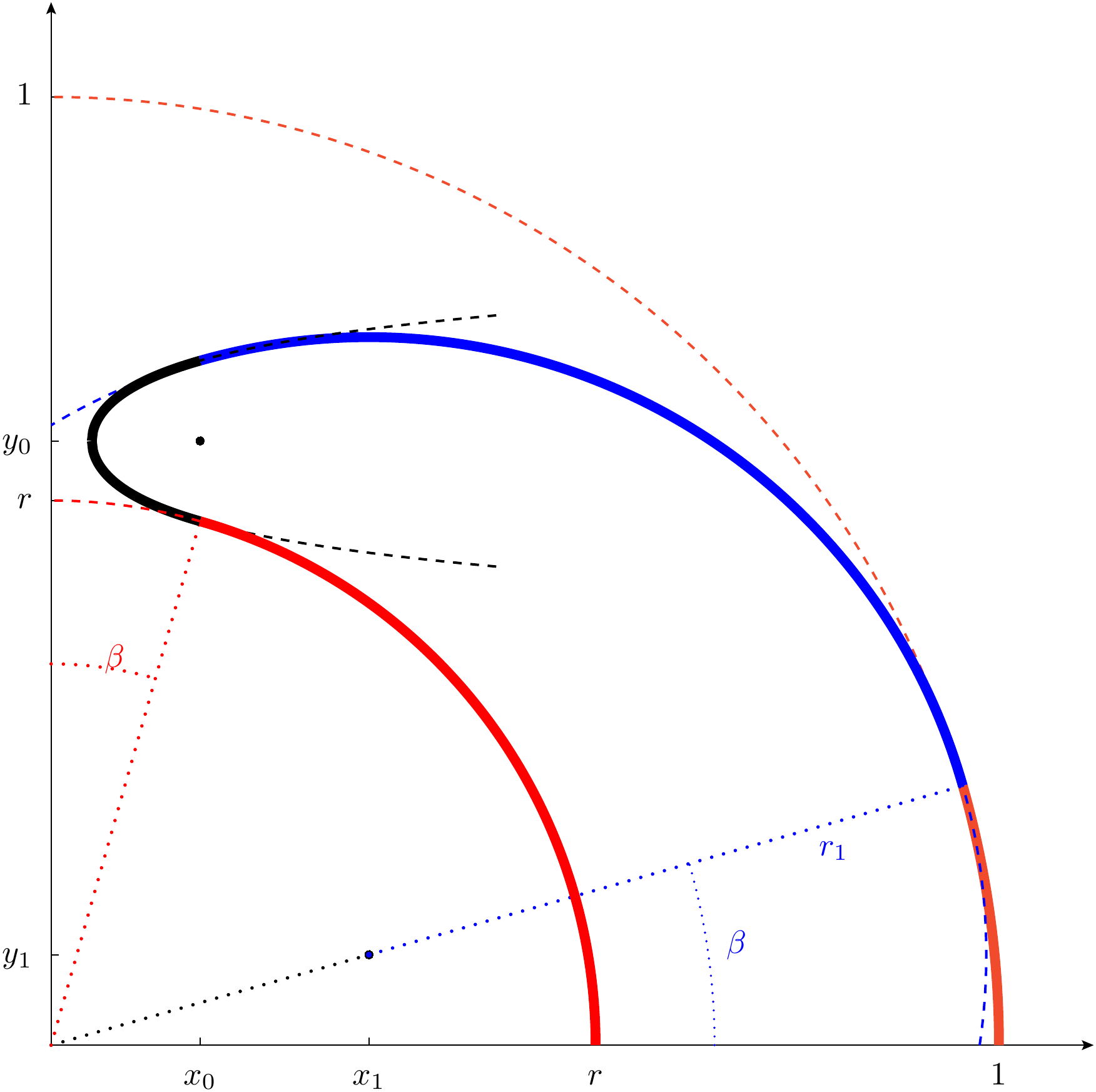}
  \caption{The construction in the upper half-space}
  \label{fig:const}
\end{figure}
This way we obtain a $C^1$ curve $\gamma_+$ in the $(x,y)$-plane by
pasting together the traces of
\begin{itemize}
\item
a curve $\gamma_1$ that parametrizes the unit circle from $(1,0)$ to
$(\cos(\beta),\sin(\beta))$  (the solid green line in Figure
\ref{fig:const}):
\begin{gather*}
  \gamma_1: (0,\beta)\,\to\, \R^2,\quad \gamma_1(s)\,=\,
  \begin{pmatrix}\cos s\\ \sin s\end{pmatrix}.
\end{gather*}
\item
a curve $\gamma_2$ that follows the circle $S_{r_1}((x_1,y_1))$ from
$(\cos(\beta),\sin(\beta))$ to
$(x_1,y_1)+r_1\big(\cos(\beta+\pi/2),\sin(\beta+\pi/2)\big)$ (the solid
blue line in Figure \ref{fig:const})  
\begin{gather}
  \gamma_2: (\beta,\beta+r_1\frac{\pi}{2})\,\to\, \R^2, \qquad
  \gamma_2(s)\,=\, \begin{pmatrix}x_1\\x_0\end{pmatrix} +
  r_1\begin{pmatrix}\cos (\beta +r_1^{-1}(s-\beta))\\
    \sin (\beta  +r_1^{-1}(s-\beta)) \end{pmatrix}
\end{gather}
\item
two curves $\gamma_3^\pm$ that describes the catenary $\{(x,y): |y - y_0| = \lambda
\arccosh(\frac{x}{\lambda})\}$ for $\lambda\leq x\leq x_0$ (the solid
black line in Figure \ref{fig:const})
\begin{gather*}
  \gamma_3: (\lambda,x_0)\,\to\, \R^2, \qquad \gamma_3(x)\,=\, 
  \begin{pmatrix} x\\ y_0 \pm  \lambda (\arccosh x/\lambda). 
  \end{pmatrix}
\end{gather*}
\item
and finally a curve $\gamma_4$ that parametrizes the circle $S_{r}(0)$ between
$r(\cos(\pi/2-\beta),\sin(\pi/2-\beta))$ and $(r,0)$ (the solid red line in Figure
\ref{fig:const}).
\begin{gather*}
  \gamma_4: (0,\frac{\pi}{2}-\beta)\,\to\R^2,\qquad \gamma_4(s)\,=\,
  r\begin{pmatrix} \cos (s/r)\\\sin (s/r)\end{pmatrix}
\end{gather*}
\end{itemize}
\step{2}
The conditions above are expressed
in the following system of equations, 
\begin{align}
  \begin{pmatrix}x_1\\y_1\end{pmatrix} +
  r_1\begin{pmatrix}\cos\beta\\\sin\beta\end{pmatrix} \,&=\,
  \begin{pmatrix}\cos\beta\\\sin\beta\end{pmatrix}, \label{eq:sys1}\\
  \begin{pmatrix}x_0\\y_0 +\lambda \arccosh (\lambda^{-1}x_0)\end{pmatrix} \,&=\,
  \begin{pmatrix}x_1\\y_1\end{pmatrix} +
  r_1\begin{pmatrix}-\sin\beta\\\cos\beta\end{pmatrix}, \label{eq:sys2}\\
  \frac{1}{x_0}\begin{pmatrix}\sqrt{x_0^2-\lambda^2}\\\lambda
  \end{pmatrix} \,&=\,
  \begin{pmatrix}\cos\beta\\\sin\beta\end{pmatrix}, \label{eq:sys3}\\
  \begin{pmatrix}x_0\\y_0-\lambda\arccosh (\lambda^{-1}x_0)\end{pmatrix} \,&=\,
  r\begin{pmatrix}\sin\beta\\\cos\beta\end{pmatrix}. \label{eq:sys4}
\end{align}
After some manipulations, and defining $F:\R^3\,\to\,\R^2$ by
\begin{gather*}
  F(r,r_1,\beta)\,=\, \begin{pmatrix}
  r\cos\beta + 2r\sin^2\beta\arccosh\frac{1}{\sin\beta}
  -r_1\cos\beta -(1-r_1)\sin\beta\\
  (r + r_1)\sin\beta - (1-r_1)\cos\beta \end{pmatrix}
\end{gather*}
we obtain the equivalent system
\begin{align}
  0\,&=\, F(r,r_1,\beta), \label{eq:rsys1}\\
  x_1\,&=\, (1-r_1)\cos\beta, \label{eq:rsys3}\\
  y_1\,&=\, (1-r_1)\sin\beta, \label{eq:rsys4}\\
  x_0\,&=\, r\sin\beta, \label{eq:rsys5}\\
  \lambda\,&=\, x_0\sin\beta, \label{eq:rsys6}\\
  y_0\,&=\,\frac{1}{2}\Big(y_1 + (r+r_1)\cos\beta\Big). \label{eq:rsys7}
\end{align}
We next observe that $F(1,1,0)=0$ and that $F$ is continuously
differentiable. Moreover we have
\begin{gather*}
  \det \Big(\partial_{r_1}F\quad \partial_\beta F\Big)(1,1,0)\neq 0.
\end{gather*}
Hence, by the Implicit Function Theorem, we obtain $C^1$-functions $r_1=r_1(r)$,
$\beta=\beta(r)$ such that $(r,r_1(r),\beta(r))$ satisfy
\eqref{eq:rsys1} for $0<r<1$ close to one. 
For the derivatives of $r_1,\beta$ with respect to $r$ we obtain that
\begin{gather}
   r_1'(1) \,=\, 1, \qquad \beta'(1)\,=\,
   -\frac{1}{2}, \label{eq:prime-r}
\end{gather}
which shows that $0<r_1<1$ and $0<\beta<\pi/2$ for $0<r<1$ close to
one. $(x_0,y_0),(x_1,y_1)$ and $\lambda$ are easily determined from
\eqref{eq:rsys3}-\eqref{eq:rsys7} and are in the range of meaningful
values with respect to our construction. In particular we obtain
\begin{gather}
  (x_0,y_0)\,\to\, (0,1), \quad (x_1,y_1)\,\to\, (0,0), \quad
  \lambda\,\to\, 0 \qquad\text{ as }r\nearrow 1. \label{eq:lim-x0y0}
\end{gather}
\step{3}
We compute the surface area of $\Sigma_+$. Let $A_i$, $i=1,...,4$ denote 
the surface area of the parts of the surface that belong to the curves
$\gamma_i$.
Since $\gamma_i$ is parametrized by arc-length for $i=1,2,4$ the
corresponding surface area elements are given by the $x$-components of
$\gamma_i$. We therefore deduce that
\begin{align}
  A_1\,&=\, 2\pi \int_0^\beta \cos s\,ds \,=\, 2\pi \sin
  \beta, \label{eq:A1}\\
  A_2\,&=\, 2\pi \int_\beta^{\beta+r_1\frac{\pi}{2}} (1-r_1)\cos\beta +
  r_1\cos(\beta +r_1^{-1}(s-\beta))\, ds \notag\\
  &=\,
  2\pi\Big(r_1(1-r_1)\frac{\pi}{2}\cos\beta + r_1^2(\cos \beta
  -\sin\beta)\Big), \label{eq:A2}\\
  A_4\,&=\, 2\pi\int_0^{\pi/2-\beta} r\cos(s/r)\,ds\,=\, 2\pi r^2
  \cos\beta.   \label{eq:A4}
\end{align}
The curve $\gamma_3$ parametrizes the upper and lower part of the
catenary as two graphs.
Since the surface area element for the rotation of a graph $x\mapsto
(x,f(x))$ around the $y$-axis is given by $x\sqrt{1+f'(x)^2}$ we obtain
that
\begin{align}
  A_3\,&=\, 2\cdot 2\pi \int_\lambda^{x_0} \Big(1 +
  \frac{\lambda^2}{x^2-\lambda^2}\Big)^{1/2}x\,dx \,=\,
  2\pi\Big(x_0\sqrt{x_0^2-\lambda^2} +
  \lambda^2\arccosh\frac{x_0}{\lambda}\Big) \notag\\
  &=\, 2\pi\Big(r^2\sin^2\beta\cos\beta
  +r^2\sin^4\beta\arccosh\frac{1}{\sin\beta}\Big).  \label{eq:A3} 
\end{align}
The surface area of $\Sigma_+$ is thus given as
\begin{align}
  A_+\,:=\,&\ar(\Sigma_+)
  =\, A_1 +A_2+  A_3 +A_4\notag\\
  \,=\, &2\pi\Big(
  (1-r_1^2)\sin\beta +r_1 (1-r_1)\frac{\pi}{2}\cos\beta + (r_1^2+r^2)\cos\beta
  + \notag\\
  &\qquad\qquad\qquad + r^2\sin^2\beta\cos\beta
  +r^2\sin^4\beta\arccosh\frac{1}{\sin\beta}\Big). \label{eq:A+}
\end{align}
If we develop $A_+=A_+(r)$ at $r=1$ we obtain $A(1)=4\pi$ and
\begin{gather}
  A'(1)\,=\, 2\pi\Big(-\frac{\pi}{2} +2\Big)\,>\,0. \label{eq:A-prime}
\end{gather}
Moreover we see that
\begin{gather}
  A_1,A_3\,\to\, 0,\quad A_2,A_4\,\to\, 2\pi\qquad\text{ as }r\nearrow 1. \label{eq:lim-Ai}
\end{gather}
\step{4}
We compute the Willmore energy of the different parts. Since all these
parts have constant mean curvature given by
$2,\frac{2}{r_1},0,\frac{2}{r}$ respectively we obtain
\begin{align}
  W_+\,:=\,\W(\Sigma_+)\,&=\, \frac{1}{2}\pi( 4A_1 + \frac{4}{r_1^2}A_2 +
  \frac{4}{r^2}A_4)\notag\\
  &=\, \frac{1}{2}\pi\Big(4\sin\beta + \Big(\frac{1-r_1}{r_1}\frac{\pi}{2}\cos\beta + (\cos \beta
  -\sin\beta)\Big) +  \cos\beta\Big) \notag\\
  &=\, \pi^2\frac{1-r_1}{r_1}\cos\beta + 4\pi\cos\beta. \label{eq:W-S+}
\end{align}
From the first line and \eqref{eq:lim-Ai} we also get that
\begin{gather}
  W_+(r)\,\to\, 4\pi\qquad\text{ as }r\nearrow 1 \label{eq:lim-Wi}
\end{gather}
and
\begin{gather}
  W_+'(1)\,=\, -\pi^2. \label{eq:div-W}
\end{gather}
\step{5}
Finally, we add the lower part of the construction. With this aim we put
${\Sigma}_-$ to be the union of the lower unit sphere and the
lower part of the sphere $S_r(0)$, where we have added an inward bump
similar to the construction in Theorem \ref{thm:m-sphere}.
We can then choose the size of a bump in such a way that
$\Sigma=\Sigma_+\cup\Sigma_-$ satisfy the area constraint 
$\ar(\Sigma)\,=\, 8\pi$ and such that $\W(\Sigma)$ is arbitrarily
close to $2\W(S^2)$.
%=======================
% bibliography
%=======================
\def\cprime{$'$}

%\bibliography{all,isop,bio,eigene}
%\bibliographystyle{plain}
%
\end{document}